\documentclass[11pt,oneside]{amsart}

\usepackage{amssymb, amsmath, amsthm}
\usepackage{mathabx}
\usepackage{graphicx}
\usepackage[utf8]{inputenc}
\usepackage{hyperref}
\usepackage{mathtools}
\usepackage{rotating}
\usepackage{tabularx}
\usepackage{dynkin-diagrams}
\usepackage{subcaption} 
\usepackage[all]{xy}
\usepackage{multirow}
\usepackage{booktabs}

\usepackage[margin=3cm, right=4cm]{geometry}

\newtheorem{thm}{Theorem}
\newtheorem*{thm*}{Theorem}
\newtheorem{cor}[thm]{Corollary}
\newtheorem{lem}[thm]{Lemma}
\newtheorem{prop}[thm]{Proposition}

\newtheorem{rem}[thm]{Remark}

\theoremstyle{definition}

\numberwithin{thm}{section}
\numberwithin{table}{section}

\newcommand{\Z}{\mathbb{Z}}
\newcommand{\N}{\mathbb{N}}
\newcommand{\weyl}{\mathcal{W}}

\newcommand{\code}{L}
\newcommand{\I}{I}

\newcommand{\rr}[1]{\mathbf{#1}}
\newcommand{\rrt}[2]{\widehat{\mathbf{#1}}_{{#2}}}
\renewcommand{\tilde}{\widetilde}

\title[The Bruhat Order via Intrinsic Coverings of Compositions]{The Bruhat Order on Symmetric Groups via Intrinsic Coverings of Compositions}

\author[lambert]{Jordan Lambert}
\address[Jordan Lambert]{Department of Mathematics, ICEx, Universidade Federal Fluminense, Volta Redonda-RJ, Brazil}
\email{jordanlambert@id.uff.br}

\author[rabelo]{Lonardo Rabelo}
\address[Lonardo Rabelo]{Department of Mathematics, Federal University of Juiz de Fora, Brazil}
\email{lonardo.rabelo@ufjf.br}

\thanks{This work was supported by the  FAPERJ (Carlos Chagas Filho Foundation for Supporting Research in the State of
Rio de Janeiro) no. 010.002602/2019.}



\begin{document}

\begin{abstract} Lehmer’s code defines a bijection between the symmetric group and the set of staircase compositions. In this paper, we characterize a poset structure on these compositions that is equivalent to the strong Bruhat order on the symmetric group. This construction is intrinsic and does not require any reference to the associated permutations.
\end{abstract}

\maketitle

\allowdisplaybreaks

\section{Introduction}

This paper aims to characterize the strong Bruhat order on the symmetric group $S_n$ via the associated Lehmer code. The code defines a bijection between $S_n$ and the set of staircase compositions $\mathcal{C}_n$, which naturally leads to the question: how does the strong Bruhat order manifest within $\mathcal{C}_n$? In this work, we will define an intrinsic partial order in $\mathcal{C}_n$ and prove its equivalence to the strong Bruhat order on $S_n$. Of particular interest, this poset is defined without any reference to the associated permutations.

The construction of the intrinsic partial order in $\mathcal{C}_n$ may be summarized as follows.
Any composition $\alpha=(\alpha_1,\ldots, \alpha_{n-1}) \in \mathcal{C}_n$ determines an unique diagram, from which we construct a matrix $c_{i,j}(\alpha)$ that encodes how the rows of the composition are stacked in the diagram. Given two compositions $\alpha$ and $\alpha'$ such that $|\alpha|=|\alpha'|+1$, we define the covering relation by imposing four conditions among their parts (see Equations \eqref{eq:a1}, \eqref{eq:a2},\eqref{eq:a3} and \eqref{eq:a4}). The simplest case occurs when the difference between $\alpha$ and $\alpha'$ lies in a unique single part, i.e., $\alpha'_i=\alpha_i-1$. However, a covering relation may also arise from simultaneous changes in two parts. In such cases, there exist indices $i<j$ such that $\alpha'_i< \alpha_i-1$, and the excess is transferred to another part, giving $\alpha'_j= \alpha_j+(\alpha_i-\alpha'_i-1')$. This process is not arbitrary: one must ensure that there is enough space to move and to accommodate the excess in the diagram. Specifically, Equation \eqref{eq:a3} ensures that the movement along the diagram is possible, and the condition $c_{i,j}(\alpha)=c_{i,j}(\alpha')=\alpha'_i-\alpha_j$, given by Equation \eqref{eq:a4}, must be satisfied.

This covering relations define a poset which is denoted by $(\mathcal{PC}_n, \preceq)$. We show that it is isomorphic to the poset on the symmetric group endowed with the strong Bruhat order (Theorem \ref{thm:codecovering2}). This result relies on an interplay between the permutation matrices and the  Extended Lehmer code, as we show it is the same as the matrix $c_{i,j}$ (Proposition \ref{prop:valueD}).
We also need to relate the existence of the covering pairs through Coxeter movings over the corresponding reduced decompositions (Proposition \ref{prop:moves_property2}). 
This is illustrated by the ladder moves over the corresponding diagrams (see Subsection \ref{subsection:diagrams}).

Finally, for a given composition $\alpha\in \mathcal{C}_n$, we present a method to identify the elements covered by $\alpha$ and those by which $\alpha$ is covered. This property is obtained by checking if $\alpha$ is $(i,z)$-removable or $(i,z)$-insertable, for any pair $(i,z)$ of positive integers representing either the removed or the inserted box of $\alpha$ (Propositions \ref{prop:remove_condition} and \ref{prop:insertion_condition}). We briefly discuss how it may be used to compute the formula of the Monk's rule (see Subsection \ref{subsec:monksrule}). 

We should remark that this problem was motivated by a geometric question concerning the Bruhat decomposition on the maximal flag varieties of $\mathrm{Sl}_n(\mathbb{R})$. The Schubert cells are parametrized by $S_n$ in a such way the computation of the incidence coefficients among them requires data coming from the determination of the covering pairs (for details see Rabelo--San Martin\cite{RSm19} and Matszangosz \cite{Mat19}).

A secundary topic explored in this work is the realization of permutations by diagrams. Our presentation resembles the construction of Young's diagrams in the context of the Grassmannian permutations. There exists a variety of types of diagrams when dealing with permutations (for example, see Manivel \cite{Man01} for the Rothe Diagrams, Coşkun--Taşkın \cite{CT18} for the Tower diagrams, and Bergeron-Billey \cite{BBil93} for the RC-graphs). We hope to present here some advantages of our choice, for example, the row-reading map that provides a direct reduced decomposition of permutation in terms of simple reflections. 

It is worth noting that our interpretation of the covering relations in terms of the composition's diagrams are similar to those obtained by Coşkun--Taşkın in the context of the Tower diagrams (Proposition \ref{prop:laddermoves} vs. \cite{CT18}, Theorem 4.1). In addition, Denoncourt \cite{Den13} also establishes a covering equivalence result in the (left) weak order using the extended Lehmer code (see \cite{Den13}, Proposition 2.8).

This work is arranged as follows.
In Section \ref{sec:preliminaries}, we introduce the main definitions of the combinatorics of the symmetric group.
In Section \ref{sec:newposet}, we define the poset for the set of staircase compositions $\mathcal{C}_n$ and derive some intrinsic properties.
In Section \ref{sec:equivalence}, we prove the equivalence between this poset in $\mathcal{C}_n$ and the strong Bruhat order of $S_n$.
To conclude, in Section \ref{sec:rem_ins_comp}, we introduce both the removing and inserting algorithms over the compositions.

\section{Preliminaries}\label{sec:preliminaries}

Let $\N=\{1,2,3, \dots\}$ and $\Z$ be the set of integers. For $n,m\in\Z$, with $n\leqslant m$, denote the set $[n,m]=\{n,n+1, \dots, m\}$. For $n\in \N$, denote $[n]=[1,n]$.

The symmetric group $S_n$, regarded as a Coxeter group of type A, is generated by simple transpositions $s_i$ for  $i\in [n-1]$ subject to the relations:
\begin{itemize}
\item $s_{i}^{2}  = 1$;
\item Commutation: $s_is_j=s_js_i$ for $|i-j|\geqslant 2$;
\item Braid: $s_is_{i+1}s_i=s_{i+1}s_is_{i+1}$ for $i\in[n-2]$.
\end{itemize}

Generically, a move refers to either a commutation or a braid relation.
Given any permutation $w \in S_n$, the length $\ell(w)$ of $w$ is the minimal number of simple transpositions needed to decompose $w$ as a product $w=s_1 \ldots s_{\ell(w)}$. Such a decomposition is called reduced. The Word Property guarantees that every two reduced decompositions for $w$ can be connected by a sequence of moves.

There is a partial order in $S_n$ which is called the (strong) Bruhat order: $u\leqslant w$ if given a reduced decomposition $w = s_{j_{1}} \cdots s_{j_{r}}$ then $u=s_{j_{i_{1}}}\cdots s_{j_{i_{k}}}$ for some $1\leqslant i_1\leqslant \cdots \leqslant i_r\leqslant r$.

If there exists $w,w' \in \weyl$ such that $w'\leqslant w$ and $\ell(w) = \ell(w')+1$ then $w$ covers $w'$ (alternatively, $w,w'$ is a covering pair). Given a reduced decomposition $w=s_{1}\cdots s_{\ell}$, if $w$ covers $w'$ then $w' = s_{1} \cdots \widehat{s_{\I}}\cdots s_{\ell}$, for some $I \in [\ell]$. The integer $\I$ depends on both $w'$ and the choice of the reduced decomposition of $w$.

It is also interesting to denote a permutation $w\in S_{n}$ in the one-line notation by $w=w(1)w(2)\cdots w(n)$. In this permutation model, the simple reflections may be viewed as transpositions $s_i=(i,i+1)$ in such a way that $s_i$ acts at right by swapping $w(i)$ and $w(i+1)$ (the values at positions $i$ and $i+1$) while $s_i$ acts at left by exchanging the values $i$ and $i+1$. An inversion of $w$ is a pair $(i,j)$ such that $i<j$ and $w(i)>w(j)$. The length $\ell(w)$ is precisely the number of the inversions of $w$.

The following lemma provides a specific characterization of the covering relation using the one-line notation.

\begin{lem}[\cite{BB05}, Lemma 2.1.4]\label{lem:bjorner}
Let $w,w'\in S_{n}$. Then, $w$ covers $w'$ in the Bruhat order if and only if $w = w' \cdot (i,j)$ for some transposition $(i,j)$ with $i<j$ such that $w'(i) < w'(j)$ and there does not exist any $k$ such that $i<k<j$, $w'(i) < w'(k) <w'(j)$.
\end{lem}

The lemma says that if $w=w(1)w(2)\cdots w(n)$ then $w'$ is covered by $w$ if and only if $w'$ is obtained from $w$ by switching the values in position $i$ and $j$, for some pair $i<j$, and such that no value between positions $i$ and $j$ lies in $[w(j),w(i)]$.

A finite integer sequence $\alpha = (\alpha_{1}, \dots, \alpha_{m})$ is a (weak) composition if $\alpha_{i} \geqslant 0$ for all $i\in [m]$. If required, we assume that $\alpha_{i}=0$ for $i>m$.
The elements $\alpha_{i}$ of the sequence are called the parts, the number $\ell(\alpha)$ of parts is the length, and the sum $|\alpha|$ of the parts is the weight of the composition. A partition is a weakly decreasing decomposition.

Define the set $\mathcal{C}_{n}$ as the set of compositions $\alpha$ such that $\alpha_{i} \leqslant n-i$, i.e., the set of elements inside the cartesian product $[0,n-1]\times [0,n-2] \times \cdots \times [0,1]$.

The Lehmer code (briefly called code) of a permutation $w\in S_{n}$ is an integer sequence $\alpha$ with $\alpha_{i} = \#\{k>i \ |\ w(k)<w(i)\}$ and it will be denoted by $\code(w)$. In other words, each entry of the code corresponds to the number of inversions to the right of $w_{i}$. Since $0\leqslant \alpha_{i}\leqslant n-i$, we have $\alpha=(\alpha_{1}, \dots, \alpha_{n-1})=\code(w) \in \mathcal{C}_n$.

Let us describe the permutation matrix of $w$. Consider an $n\times n$ array of boxes with rows and columns indexed by integers $[n]$ in matrix style. The \emph{permutation matrix} associated to a permutation $w\in S_n$ is obtained by placing dots in positions $(w(i),i)$, for all $1\leqslant i\leqslant n$, in the array.
The code $\alpha=\code(w)$ admits the following interpretation: $\alpha_{i}$ is the number of dots in the region strictly above and to the right of the dot in the $i$-th column of the permutation matrix.

This data allows us to represent any permutation $w\in S_{n}$ as a diagram inside a staircase shape $(n-1) \times (n-1)$ by pilling up the parts $\alpha_i$ of the composition $\alpha$. In this way, the diagram of $w$ is the collection of left-justified boxes where the $i$-th row counted from bottom to top contains $\alpha_{i}$ boxes \footnote{The diagram of an element $w$ corresponds to the bottom RC-graph of $w$. It may be also given as a left-justified version of the corresponding Rothe diagram}. 

Indeed, the Lehmer code provides a bijection between $S_n$ and $\mathcal{C}_n$.

\begin{lem}[\cite{Man01}, Proposition 2.1.2]\label{lem:perm_by_diagram}
A permutation is determined by its code and, therefore, by its diagram. 
\end{lem}

For instance, consider $w=5\,7\, 6 \,2\,1\,8\,3\,4 \in S_{8}$. The code $\code(w) =(4,5,4,1,0,2,0)$ is represented by its diagram in Figure \ref{fig:ex_diag} (left). In the sequence, the corresponding permutation matrix of $w$ in Figure \ref{fig:ex_diag}(right).

\begin{figure}[hbtp]
\centering
\includegraphics[scale=1]{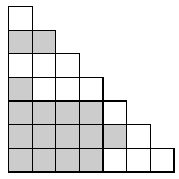}
\hspace{1cm}
\includegraphics[scale=1]{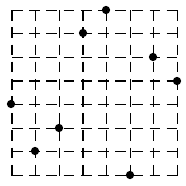}
\caption{On the left, diagram of $w=5\,7\, 6 \,2\,1\,8\,3\,4$. On the right, its permutation matrix.}
\label{fig:ex_diag}
\end{figure}

If one has a composition (the code), there is a nice way to illustrate the procedure described by Manivel (\cite{Man01}, Chapter 2) to determine the permutation in the diagram. For each row $i$, beginning from the bottom to the top, write a path along the bottom edges of the boxes starting from the left and finishing with the right side edge at the last box, such that $w(i)$ is the label of the last top arrow of the path assigning the numbers from $1$ to $n$ omitting $w(1), \ldots, w(i-1)$. For instance, the Figure \ref{fig:ex_comp_diag} shows the first three steps to recover the permutation whose composition is $\alpha =(4,5,4,1,0,2,0)$.

\begin{figure}[hbtp]
\centering
\includegraphics[scale=1]{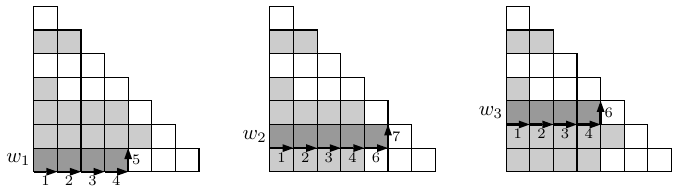}
\caption{Recovering the permutation from $\alpha=(4,5,4,1,0,2,0)$.}
\label{fig:ex_comp_diag}
\end{figure}

This arrangement of boxes provides an easy way to express the permutation $w$ in terms of simple reflections $s_{i}$, which we will call the row-reading expression of $w$. We begin assigning a simple reflection consecutively to each box from left to right and upwards, starting from $s_{1}$ in the bottom leftmost box in the staircase shape. Then, we obtain a reduced decomposition by reading each row in the diagram from right to left, and the rows from bottom to top. More specifically, define the reading for the $i$-th row by $\rr{w}_{i}= s_{\alpha_{i}+i-1} \cdot s_{\alpha_{i}+i-2} \cdots s_{i+1} \cdot s_{i}$ if $\alpha_{i}$ is non-zero, and $\rr{w}_{i} = e$ otherwise. The row-reading expression $\rr{w}$ of $w$ is
\begin{equation*}
\rr{w} = \rr{w}_{1}\cdots \rr{w}_{n-1}.
\end{equation*}

Notice this row-reading resembles that described by Manivel (see \cite{Man01}, Remark 2.1.9). For instance, the row-reading expression of $w=5\,7\, 6 \,2\,1\,8\,3\,4$ is $\rr{w}=s_{4} s_{3} s_{2} s_{1}\cdot s_{6}s_{5}s_{4}s_{3}s_{2}  \cdot s_{6}s_{5}s_{4}s_{3} \cdot s_{4} \cdot s_{7}s_{6}$ which can be obtained from Figure \ref{fig:rowreading} by reading each row from right to left, beginning from the bottom.

\begin{figure}[ht]
\centering
\includegraphics[scale=1]{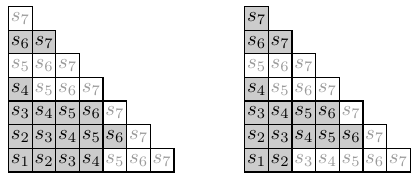}
\caption{The row-reading of $w=5\,7\, 6 \,2\,1\,8\,3\,4$ on the left and of $w' = 3\,7\, 6 \,2\,1\,8\,5\,4$ on the right.}
\label{fig:rowreading}
\end{figure}

Recall that if $w$ covers $w'$ such that $w=s_{1}\cdots s_{\ell}$ is a reduced decomposition then there is an integer in $\I\in[\ell]$ such that $w' = s_{1} \cdots \widehat{s_{\I}}\cdots s_{\ell}$.
Let us denote by $\rrt{w}{\I}$ this reduced decomposition of $w'$ obtained by removing the $\I$-th simple reflection from the row-reading $\rr{w}$.

For instance, by Lemma \ref{lem:bjorner}, $w = 5\,7\, 6 \,2\,1\,8\,3\,4$ covers $w' = \mathbf{3}\,7\, 6 \,2\,1\,8\,\mathbf{5}\,4$. Moreover, $\rrt{w}{2}=s_{4} \widehat{s_{3}} s_{2} s_{1}\cdot s_{6}s_{5}s_{4}s_{3}s_{2}  \cdot s_{6}s_{5}s_{4}s_{3} \cdot s_{4} \cdot s_{7}s_{6}$ is a reduced decomposition of $w'$ obtained from $\rr{w}$. However, it differs from the row-reading $\rr{w'}$ since $\code(w') = (2,5,4,1,0,2,1)$ and $\rr{w'}=s_{2} s_{1}\cdot s_{6}s_{5}s_{4}s_{3}s_{2}  \cdot s_{6}s_{5}s_{4}s_{3} \cdot s_{4} \cdot s_{7}s_{6}\cdot s_{7}$ (cf. Figure \ref{fig:rowreading}).
The Word Property guarantees that there is a sequence of moves that transforms $\rrt{w}{2}$ into $\rr{w'}$; in this case, $s_{4}$ in the first row turns into $s_{7}$ in the seventh row after a sequence of moves.

There is a known characterization of covering relations in terms of the one-line notation of permutations $w$ and $w'$ (see Lemma~\ref{lem:bjorner}). However, our goal is to establish a relationship that also takes into account the reduced decompositions, alongside the permutation structure. Since the Lehmer code of a permutation encodes both the length and information about its reduced decompositions, we aim to characterize covering relations directly in terms of these codes.

As a motivation, consider the following situation. Let $w=5\,7\, 6 \,2\,1\,8\,3\,4 \in S_8$ with $\code(w) = (4,5,4,1,0,2,0)$. As a direct application of Lemma \ref{lem:bjorner}, it is immediate that $w$ covers both $w'$ $=3\,7\, 6 \,2\,1\,8\,5\,4$ and $w''$ $=5\,7\, 4 \,2\,1\,8\,3\,6$, where $w=w'\cdot(1,7)$ and $w=w''\cdot(3,8)$. Notice also that $\alpha'=\code(w') = (2,5,4,1,0,2,1)$ and $\alpha''=\code(w'')= (4,5,3,1,0,2,0)$. For the pair $w,w''$, the only distinction between their codes is $\alpha''_3=\alpha_3-1$, while for the pair $w,w'$, the distinction is $\alpha'_1=\alpha_1-2$ and $\alpha''_7=\alpha_{7}+1$. Hence, for the pair $w,w''$, their codes differ only at one position with $\alpha''_i$ equals $\alpha_i-1$. However, for the pair $w,w^{'}$, their codes differ at two positions and it is not clear a priori how the values in those positions are related. It reveals that the Bruhat order given in terms of the Lehmer code sometimes is manifested by changes in two positions of the code.

What happens for the pair $w,w''$ is justified by the following lemma.

\begin{lem}\label{lem:covering}
Let $w,w'\in S_{n}$ and denote by $\alpha = \code(w)$ and $\alpha' = \code(w')$. If there exists $i$ such that $\alpha_{i}' = \alpha_{i} - 1$ and $\alpha_{k}' = \alpha_{k} \mbox{ for every } k \neq i$ then $w$ covers $w'$.
\end{lem}
\begin{proof}
Let us show that for some $I\in [\ell(w)]$ the reduced decompositions $\rrt{w}{I}$ and $\rr{w'}$ of $w'$ are equal.
We have that $\rr{w}_{i}=s_{\alpha_{i}+i-1}\cdot \rr{w}_{i}'$, and $\rr{w}_{k}=\rr{w}_{k}'$ for every $k\neq i$. Then, for $I =\left(\sum_{k=1}^{i-1} \alpha_{k}\right) +1$, we have $\rr{w'} = \rr{w}_{1}'\cdots \rr{w}_{i-1}' \cdot \widehat{s}_{\alpha_{i}+i-1} \cdot \rr{w}_{i}'\cdots \rr{w}_{n-1}'= \rrt{w}{I}$.
\end{proof}

As it was already noticed, the pair $w,w'$ in the above example shows that the converse of Lemma \ref{lem:covering} isn't true. In the next section \ref{sec:newposet}, we will introduce a poset in the set $\mathcal{C}_{n}$ that will solve this question. In Section \ref{sec:equivalence} we will show that such poset is equivalent to the (strong) Bruhat order in $S_{n}$.

\section{An intrinsic covering relation for compositions}\label{sec:newposet}

In this section, we describe a covering relation for compositions that provides a poset structure for $\mathcal{C}_{n}$. We highlight that all results obtained in this section don't require any mention of the permutation, using only information available from compositions.

Let $\alpha=(\alpha_{1},\dots, \alpha_{m})$ be a composition in $\N^{m}$. Consider $N(\alpha)$ the integer given by  $N(\alpha)=\max_{\alpha_{i}\neq 0}\{\alpha_{i}+i\}$. Notice that $\alpha$ lies in $\mathcal{C}_{N(\alpha)}$. Moreover, $N(\alpha)$ is the smaller with such property.

Given $i\in[m]$ and $j \in \N$, define $c_{i,j}(\alpha)$ to be the integer defined recursively with respect to $j$ as follows:
\begin{itemize}
\item If $j\leqslant i+1$, then $c_{i,j}(\alpha) = 0$;
\item If $j>i+1$, then $c_{i,j}(\alpha) = c_{i,j-1}(\alpha) + \left\{
\begin{array}{cl}
1, & \mbox{if } \alpha_{j-1} < \alpha_{i} - c_{i,j-1}(\alpha);\\
0, & \mbox{if } \alpha_{j-1} \geqslant \alpha_{i} - c_{i,j-1}(\alpha).
\end{array}\right.$
\end{itemize}

This defines a matrix $c(\alpha)$ with infinite entries.
For instance, $c(\alpha)$ for the composition $\alpha = (4,5,4,1,0,2,0)$ is  
\begin{equation*}
c(\alpha) =
\left[
\begin{array}{ccccccccccc}
0 & 0 & 0 & 0 & 1 & 2 & 2 & 3 & 4 & 4 &\cdots \\ 
0 & 0 & 0 & 1 & 2 & 3 & 3 & 4 & 5 & 5 &\cdots \\ 
0 & 0 & 0 & 0 & 1 & 2 & 2 & 3 & 4 & 4 &\cdots \\ 
0 & 0 & 0 & 0 & 0 & 1 & 1 & 1 & 1 & 1 &\cdots \\ 
0 & 0 & 0 & 0 & 0 & 0 & 0 & 0 & 0 & 0 &\cdots \\ 
0 & 0 & 0 & 0 & 0 & 0 & 0 & 1 & 2 & 2 &\cdots \\ 
0 & 0 & 0 & 0 & 0 & 0 & 0 & 0 & 0 & 0 &\cdots
\end{array} 
\right]
\end{equation*}

The quantity $c_{i,j}(\alpha)$ can be interpreted in terms of certain paths within the diagram of $\alpha$. For each fixed index $i$, we construct a polygonal path in the diagram of $\alpha$ according to the following steps:
\begin{description}
\item[Step 1] Start at the last box of the $i$-th row;
\item[Step 2] If the box directly above is empty, move diagonally to the next box northeast of the current position;
\item[Step 3] If the box directly above is filled, move vertically to the next box directly above;
\item[Step 4] Repeat this process for each row until the top of the diagram is reached.
\end{description}

The value $c_{i,j}(\alpha)$ is then the number of steps the polygonal path moves to the left from the $i$-th row up to the $(j-1)$-th row. For instance, Figure \ref{fig:exampleD} illustrates this process for the composition $\alpha = (4,5,4,1,0,2,0)$, showing how to compute $c_{i,j}(\alpha)$.

\begin{figure}[ht]
\centering
\includegraphics[scale=0.8]{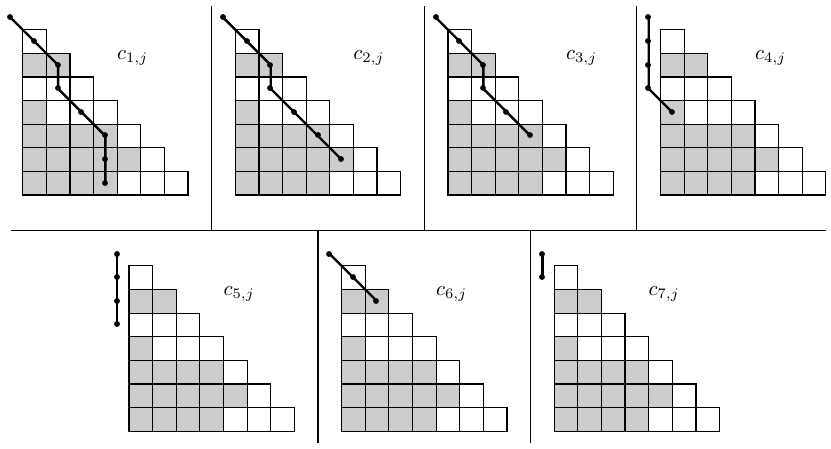}
\caption{Geometric interpretation of $c_{i,j}(\alpha)$ in the diagram of $\alpha = (4,5,4,1,0,2,0)$.}
\label{fig:exampleD}
\end{figure}

The following lemma determines some bounds for the entries of $c(\alpha)$. 

\begin{lem}\label{lem:c_properties} Let $\alpha=(\alpha_{1},\dots, \alpha_{m})$ be  a composition. Then,
\begin{enumerate}
\item $0\leqslant c_{i,j}(\alpha) \leqslant c_{i,j+1}(\alpha) \leqslant \alpha_{i}$ for every $i$ and $j$;
\item $\max\{0,\min\{\alpha_{i},\alpha_{i}+j-N(\alpha)-1\}\} \leqslant c_{i,j}(\alpha) \leqslant \min\{\alpha_{i},j-i-1\}$;
\item $N(\alpha)+1$ is the smallest index for $j$ such that $c_{i,j}(\alpha) = \alpha_{i}$ for every $i$.
\end{enumerate}
\end{lem}
\begin{proof}
Assume that $I$ is an index satisfying $\alpha_{I}\neq 0$ and $\alpha_{I}+I=N(\alpha)$.

For assertion (1), it follows directly by definition that $0\leqslant c_{i,j}(\alpha) \leqslant c_{i,j+1}(\alpha)$.
If $c_{i,j}(\alpha)=\alpha_{i}$ then $\alpha_{i} - c_{i,j}(\alpha) = 0 \leqslant \alpha_{j}$. Hence, by definition, $c_{i,j+1}(\alpha) = c_{i,j}(\alpha) +0=\alpha_{i}$. This proves that $c_{i,j}(\alpha)\leqslant\alpha_{i}$.

Let us prove affirmation (2). By definition of $c_{i,j}$ and (1), $c_{i,j}(\alpha)\leqslant \min\{\alpha_{i}, j-i-1\}$.

Now, we will prove that $c_{i,j}(\alpha) \geqslant \min\{\alpha_{i},\alpha_{i}+j-N(\alpha)-1\}$.
If $\alpha_{i} \leqslant \alpha_{i}+j-N(\alpha)-1$ then the inequality holds trivially.

Suppose that $\alpha_{i}+j-N(\alpha)-1<\alpha_{i}$, i.e., $j<N(\alpha)+1$. Let us prove by induction in $j$ that $\alpha_{i}+j-c_{i,j}(\alpha)\leqslant  N(\alpha)+1$. If $j = i+1$ then $\alpha_{i}+(i+1)-c_{i,i+1}(\alpha) = \alpha_{i}+i+1\leqslant N(\alpha)+1$.

Assume that $\alpha_{i}+j-c_{i,j}(\alpha)\leqslant N(\alpha)+1$. If $\alpha_{j}<\alpha_{i}-c_{i,j}(\alpha)$ then it follows by definition of $c_{i,j+1}$ and inductive hypothesis that $\alpha_{i}+(j+1)-c_{i,j+1}(\alpha) = \alpha_{i}+j-c_{i,j}(\alpha)  \leqslant  N(\alpha)+1$. On the other hand, if $\alpha_{j}\geqslant \alpha_{i}-c_{i,j}(\alpha)$ then $\alpha_{i}+(j+1)-c_{i,j+1}(\alpha) = \alpha_{i}+j+1-c_{i,j}(\alpha) \leqslant \alpha_{j}+j+1 \leqslant N(\alpha)+1$.

We conclude that $c_{i,j}(\alpha) \geqslant \min\{\alpha_{i},\alpha_{i}+j-N(\alpha)-1\}$. Since $c_{i,j}(\alpha)$ is non-negative, we get assertion (2).

To prove affirmation (3), apply $j = N(\alpha) +1$ into (2) to get $\alpha_{i}\leqslant c_{i,N(\alpha)+1}(\alpha)$ for every $i$. Finally, such $j$ is minimal because $c_{I,N(\alpha)}(\alpha) \leqslant\min\{\alpha_{I},N(\alpha) - I - 1\}=\alpha_{I}-1$.
\end{proof}

Given $\alpha\in \mathcal{C}_{n}$, we have that $N(\alpha) \leqslant n$. Then, Lemma \ref{lem:c_properties} allows us to describe $c(\alpha)$ as an $(n-1)\times (n+1)$ matrix for every $\alpha \in \mathcal{C}_{n}$, where we may add a few repeated columns to the right or zero rows below it.

The next lemma establishes some conditions for comparing $c_{i,j}$ for some range of $j$.

\begin{lem}\label{lem:compare_comp}
Suppose that $\alpha$ and $\tilde\alpha$ are compositions satisfying that $\tilde{\alpha}_{i}\leqslant \alpha_{i}$ and there exists $l>i$ such that $\tilde\alpha_{k}\geqslant\alpha_{k}$ for every $k\in [i+1,l]$. Then, $c_{i,k}(\tilde\alpha)\leqslant c_{i,k}(\alpha)$, for any $k\in [i+1,l+1]$.
\end{lem}
\begin{proof}
We will prove inductively for $k\in [i+1,l+1]$. If $k=i+1$ then it is trivially true by definition. Assume that the claim is true for $k\in [i+2,l]$. We can consider the following two different cases: either $c_{i,k}(\alpha) - c_{i,k}(\tilde\alpha)>\alpha_{i}-\tilde\alpha_{i}$ or $c_{i,k}(\alpha) - c_{i,k}(\tilde\alpha)\leqslant\alpha_{i}-\tilde\alpha_{i}$.

Suppose that $c_{i,k}(\alpha) - c_{i,k}(\tilde\alpha)>\alpha_{i}-\tilde\alpha_{i}$. Then, $c_{i,k}(\tilde\alpha) < c_{i,k}(\alpha)$ since $\tilde\alpha_{i}\leqslant\alpha_{i}$. Hence, $c_{i,k+1}(\tilde\alpha) \leqslant c_{i,k+1}(\alpha)$.

Suppose that $c_{i,k}(\alpha) - c_{i,k}(\tilde\alpha)\leqslant\alpha_{i}-\tilde\alpha_{i}$. Consider the following cases:
\begin{itemize}
\item If  $c_{i,k+1}(\alpha) = c_{i,k}(\alpha)+1$ then $c_{i,k+1}(\alpha)=c_{i,k}(\alpha)+1 \geq c_{i,k}(\tilde{\alpha})+1\geq c_{i,k+1}(\tilde{\alpha})$;

\item If $c_{i,k+1}(\tilde\alpha) = c_{i,k}(\tilde\alpha)$ then $c_{i,k+1}(\tilde\alpha) = c_{i,k}(\tilde\alpha)\leq  c_{i,k}(\alpha) \leq c_{i,k+1}(\alpha)$;

\item If $c_{i,k+1}(\alpha) = c_{i,k}(\alpha)$ then $\tilde\alpha_{k} \geqslant \alpha_{k} \geqslant \alpha_{i}- c_{i,k}(\alpha) \geqslant \tilde\alpha_{i}- c_{i,k}(\tilde\alpha)$. Thus, $c_{i,k+1}(\tilde\alpha) = c_{i,k}(\tilde\alpha)\leqslant c_{i,k}(\alpha) = c_{i,k+1}(\alpha)$;

\item If $c_{i,k+1}(\tilde\alpha) = c_{i,k}(\tilde\alpha)+1$ then $\alpha_{k} \leqslant \tilde\alpha_{k} < \tilde\alpha_{i}- c_{i,k}(\tilde\alpha) \leqslant \alpha_{i}- c_{i,k}(\alpha)$. Thus, $c_{i,k+1}(\alpha) = c_{i,k}(\alpha)+1\geqslant c_{i,k}(\tilde\alpha)+1=c_{i,k+1}(\tilde{\alpha})$. \qedhere
\end{itemize}
\end{proof}

Given two compositions $\alpha$ and $\alpha'$ such that $|\alpha|=|\alpha'|+1$, we say that $\alpha$ covers $\alpha'$ if there exist positive integers $i<j$ satisfying the following four conditions:
\begin{align}
\alpha_{i}' &\leqslant \alpha_{i} - 1;\label{eq:a1}\tag{a1}\\
\alpha_{j}' & =  \alpha_{j} + \alpha_{i}  -\alpha_{i}'  - 1;\label{eq:a2}\tag{a2}\\
\alpha_{k}' & = \alpha_{k} \mbox{ for every } k \neq i \mbox{ and } k \neq j;\label{eq:a3}\tag{a3}\\
c_{i,j}(\alpha) & = c_{i,j}(\alpha') = \alpha_{i}'- \alpha_{j}.\label{eq:a4}\tag{a4}
\end{align}

We can emphasize the pair $(i,j)$ saying that $\alpha$ covers $\alpha'$ in positions $(i,j)$.

For instance, suppose that $\alpha = (4,5,4,1,0,2,0)$ and $\alpha' = (2,5,4,1,0,2,1)$. Then, we can see that $i=1$ and $j=7$ satisfy conditions \eqref{eq:a1}, \eqref{eq:a2}, and \eqref{eq:a3}. We also have that $c_{1,j}(\alpha')=(0,0,0,0,1,2,2,2,2)$. Then, $c_{1,7}(\alpha') = c_{1,7}(\alpha) = 2 = \alpha_{1}'-\alpha_{7}$. We conclude that $\alpha$ covers $\alpha'$ in positions $(1,7)$.

Notice that the existence of covering pairs provides a partial order relation for $\mathcal{C}_n$ by transitivity (for details, see Stanley \cite{Sta11}, Chapter 3). Denote by $\mathcal{PC}_{n}$ the poset of $\mathcal{C}_{n}$ given by the covering relations defined above. We will denote this order relation by $\preceq$ to distinguish it from $\leqslant$ of the product order inherited from $\mathbb{N}^m$.

If $\alpha$ covers $\alpha'$ in positions $(i,j)$, we observe the pair $\alpha$ and $\alpha'$ satisfy the conditions of Lemma \ref{lem:compare_comp} from which follows that $c_{i,k}(\alpha') \leqslant c_{i,k}(\alpha)$ for any $k>i$. However, conditions \eqref{eq:a1} to \eqref{eq:a4} provide more information to compare both $c_{i,k}(\alpha')$ and $c_{i,k}(\alpha)$, as we describe in the next proposition.

\begin{prop}\label{prop:moves_property}
Suppose that $\alpha$ covers $\alpha'$ in position $(i,j)$. Then,
\begin{enumerate}
\item[(i)] $c_{i,k}(\alpha) = c_{i,k}(\alpha')$ for every $k\in [i+1,j]$;
\item[(ii)] $c_{i,j+1}(\alpha)=1+c_{i,j}(\alpha)$ and $c_{i,j+1}(\alpha')=c_{i,j}(\alpha')$;
\item[(iii)] $c_{i,k}(\alpha) > c_{i,k}(\alpha')$, for every $k>j$;
\item[(iv)] Either $\alpha_{k} +c_{i,k}(\alpha) \geqslant \alpha_{i}$ or $\alpha_{k} + c_{i,k}(\alpha) < \alpha_{i}'$, for every $k\in [i+1,j-1]$.
\end{enumerate}
\end{prop}
\begin{proof}
(i): Fix $i$. Let us show inductively that for any $k\geqslant i+1$
\begin{equation}\label{eq:lemaACover}
0\leqslant c_{i,k}(\alpha) - c_{i,k}(\alpha')\leqslant c_{i,k+1}(\alpha) - c_{i,k+1}(\alpha') \leqslant\alpha_{i}-\alpha_{i}'.
\end{equation}

If $k=i+1$ then, by definition, $0 = c_{i,i+1}(\alpha) - c_{i,i+1}(\alpha') \leqslant c_{i,i+2}(\alpha) - c_{i,i+2}(\alpha') \leqslant 1 \leqslant \alpha_{i}-\alpha_{i}'$, where the first inequality follows from the fact that if $c_{i,i+2}(\alpha')=1$ then $c_{i,i+2}(\alpha)=1$.

Assume that $k>i+1$ and the claim is true for $k$. Suppose that $c_{i,k+2}(\alpha) = c_{i,k+1}(\alpha)$. Notice that \eqref{eq:a1}, \eqref{eq:a2}, and \eqref{eq:a3} imply that $\alpha_{k+1}' \geqslant \alpha_{k+1}$. It follows by induction that $\alpha_{k+1}' \geqslant \alpha_{k+1} \geqslant \alpha_{i}- c_{i,k+1}(\alpha) \geqslant \alpha_{i}'- c_{i,k+1}(\alpha')$. Thus, $c_{i,k+2}(\alpha') = c_{i,k+1}(\alpha')$ and $c_{i,k+1}(\alpha) - c_{i,k+1}(\alpha')= c_{i,k+2}(\alpha) - c_{i,k+2}(\alpha')$.

Now, suppose that $c_{i,k+2}(\alpha) = c_{i,k+1}(\alpha)+1$. Since either $c_{i,k+2}(\alpha') = c_{i,k+1}(\alpha')$ or $c_{i,k+2}(\alpha') = c_{i,k+1}(\alpha')+1$, we have $c_{i,k+1}(\alpha) - c_{i,k+1}(\alpha')\leqslant c_{i,k+2}(\alpha) - c_{i,k+2}(\alpha')$.

For the second inequality, notice that the difference increases at most one. Thus, if $c_{i,k+1}(\alpha) - c_{i,k+1}(\alpha')<\alpha_{i}-\alpha_{i}'$ then $c_{i,k+2}(\alpha) - c_{i,k+2}(\alpha')\leqslant\alpha_{i}-\alpha_{i}'$. It remains to prove that if $c_{i,k+1}(\alpha) - c_{i,k+1}(\alpha')=\alpha_{i}-\alpha_{i}'$ then $c_{i,k+2}(\alpha) - c_{i,k+2}(\alpha')=\alpha_{i}-\alpha_{i}'$. Since $c_{i,j}(\alpha) - c_{i,j}(\alpha') = 0$, it follows that $k+1 \neq j$. Thus, $\alpha_{k+1} < \alpha_{i}- c_{i,k+1}(\alpha)$ if and only if $\alpha_{k+1}' < \alpha_{i}'- c_{i,k+1}(\alpha')$. Hence, $c_{i,k+2}(\alpha) - c_{i,k+2}(\alpha')=\alpha_{i}-\alpha_{i}'$.

Finally, since $c_{i,j}(\alpha) - c_{i,j}(\alpha') = 0$, by Equation \eqref{eq:lemaACover}, we have $c_{i,k}(\alpha) - c_{i,k}(\alpha') = 0$ for every $k\in [i+1,j-1]$.

(ii): It follows by conditions \eqref{eq:a1} and \eqref{eq:a4} that $\alpha_{i}-c_{i,j}(\alpha) = \alpha_{i}-\alpha_{i}'+\alpha_{j}>\alpha_{j}$. Thus, $c_{i,j+1}(\alpha) = c_{i,j}(\alpha)+1$. On the other hand, by conditions \eqref{eq:a1}, \eqref{eq:a2} and \eqref{eq:a4}, $\alpha_{i}'-c_{i,j}(\alpha')=\alpha_{j} \leqslant \alpha'_{j}$, by which $c_{i,j+1}(\alpha') = c_{i,j}(\alpha')$.

(iii): By (ii) and Equation \eqref{eq:lemaACover}, we have $c_{i,m}(\alpha)-c_{i,m}(\alpha') \geqslant 1$ for every $m>j$.

(iv): Suppose there exists $k \in [i+1,j-1]$ such that either $\alpha_k+c_{i,k}(\alpha)<\alpha_i$ and $\alpha_k+c_{i,k}(\alpha)\geq \alpha'_{i}$. It follows by definition that $c_{i,k+1}(\alpha)=c_{i,k}(\alpha)+1$ and $c_{i,k+1}(\alpha')=c_{i,k}(\alpha')$ which contradicts (i).
\end{proof}

For instance, suppose that $\alpha = (4,5,4,1,0,2,0)$ and $\alpha'= (2,5,4,1,0,2,1)$ as before such that $\alpha$ covers $\alpha'$ in position $(1,7)$. The Figure \ref{fig:comparingcij} illustrates the paths of $c_{1,j}(\alpha)$ and $c_{1,j}(\alpha')$ which are parallel to each other until they reach the $7$-th row, reflecting the equations of Proposition \ref{prop:moves_property}(i),(ii). Besides, the ``distance'' between the paralell paths is related to the condition \eqref{eq:a4}.

\begin{figure}[ht]
\centering
\includegraphics[scale=0.6]{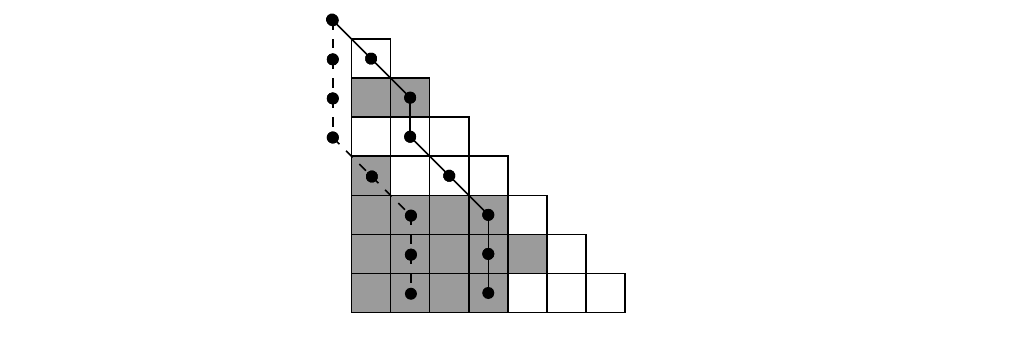}
\caption{The paths for $c_{1,j}(\alpha)$ and $c_{1,j}(\alpha')$.}
\label{fig:comparingcij}
\end{figure}

\section{The equivalence with the Bruhat order}\label{sec:equivalence}

This section establishes our main result relating the partial order in the poset \texorpdfstring{$\mathcal{PC}_{n}$}{PCn} of compositions with the strong Bruhat order in $S_{n}$. Recall that the definition of $c(\alpha)$ depends only on the composition $\alpha$. However, given a permutation $w\in S_{n}$ we are able to compute $c(\alpha)$ in a different way when $\alpha$ is the code of $w$.

\subsection{The poset isomorphism}

We begin establishing the meaning of the $c_{i,j}$'s in terms of the permutation model. 

\begin{prop}\label{prop:valueD}
Let $w\in S_{n}$ and denote by $\alpha = \code(w)$. Then, for $1\leqslant i<j\leqslant n$,
\begin{equation}
c_{i,j}(\alpha) = \#\{k \colon i<k<j \mbox{ and } w(k)<w(i) \}.\label{eq:cw}
\end{equation}
Moreover, $c_{i,n}(\alpha) = \alpha_{i}$ for $i\in[n]$.
\end{prop}
\begin{proof}
Denote by $d(A_{j}) = \#\{k \colon i<k<j \mbox{ and } w(k)<w(i) \}$. We will prove that $d(A_{j})=c_{i,j}(\alpha)$ by induction on $j$. If $j = i+1$ then $d(A_{i+1})= 0 = c_{i,i+1}(\alpha)$. Assume that $d(A_{j})= c_{i,j}(\alpha)$.

Consider the regions in the permutation matrix as in Figure \ref{fig:matrix_regions1}. Denote by $d(X)$ the number of dots in the respective region $X$.

\begin{figure}[ht]
\centering
\includegraphics[scale=1]{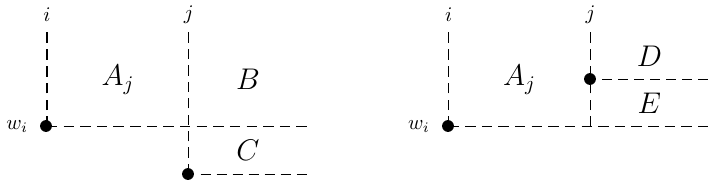}
\caption{Regions in the permutation matrix of $w$. The dots in columns $i$ and $j$ does not belong to any region. On the left $w_{j}>w_{i}$. On the right $w_{j}<w_{i}$.}
\label{fig:matrix_regions1}
\end{figure}

If $w_{j}>w_{i}$ then $\alpha_{i} = d(A_{j}) + d(B)$ and $\alpha_{j} = d(B)+d(C)$. In this case, $\alpha_{j} \geqslant \alpha_{i}-d(A_{j})=\alpha_{i}-c_{i,j}(\alpha)$ and, by definition and induction, $c_{i,j+1}(\alpha) = c_{i,j}(\alpha)=d(A_{j})=d(A_{j+1})$.

If $w_{j}<w_{i}$ then $\alpha_{i} = d(A_{j})+1 + d(D) + d(E)$ and $\alpha_{j}=d(D)$. In this case, $\alpha_{j} < \alpha_{i} - d(A_{j}) = \alpha_{i}-c_{i,j}(\alpha)$ and, by definition and induction, $c_{i,j+1}(\alpha) = c_{i,j}(\alpha)+1 = d(A_{j})+1=d(A_{j+1})$. 
\end{proof}

\begin{rem} Denoncourt in \cite{Den13} called this matrix Extended Lehmer code and denoted it by $c(w)$. His definition coincides with Equation \eqref{eq:cw}.
\end{rem}

\begin{prop}\label{prop:moves_property2}
Let $w,w'\in S_{n}$ and denote by $\alpha = \code(w)$ and $\alpha' = \code(w')$.
Suppose that $\alpha$ covers $\alpha'$ in position $(i,j)$ such that $\alpha_{i}>\alpha_{i}'+1$. Then, for every $k\in [i+1,j-1]$ and $m\in [\alpha_{i}'+1,\alpha_{i}-1]$,
\begin{align}
s_{m+i} \cdot \rr{w}_{i}' &= \rr{w}_{i}' \cdot s_{m+i}\label{eq:lemaproperties1};\\
s_{m+k-1 - c_{i,k}(\alpha)} \cdot \rr{w}_{k} &= \rr{w}_{k} \cdot s_{m+k - c_{i,k+1}(\alpha)}.\label{eq:lemaproperties2}
\end{align}
\end{prop}
\begin{proof}
Since $m>\alpha_{i}'$, we have that $s_{m+i}$ commutes with each simple reflection in $\rr{w}_{i}' = s_{\alpha_{i}'+i-1}\cdots s_{i}$. This proves Equation \eqref{eq:lemaproperties1}.

Let us prove Equation \eqref{eq:lemaproperties2}. By item (iv) of Proposition \ref{prop:moves_property}, consider the following two cases:
\begin{itemize}
\item Suppose that $\alpha_{k} + c_{i,k}(\alpha) < \alpha_{i}'$. Then $ c_{i,k+1}(\alpha) =  c_{i,k}(\alpha)+1$ and $m+k-1 - c_{i,k}(\alpha)\geqslant \alpha_{i}'+k - c_{i,k}(\alpha)>\alpha_{k}+k$. The simple reflection $s_{m+k-1 - c_{i,k}(\alpha)}$ commutes with each simple reflection in $\rr{w}_{k}$, i.e., $s_{m+k-1 - c_{i,k}(\alpha)} \cdot \rr{w}_{k} = \rr{w}_{k} \cdot s_{m+k-1 - c_{i,k}(\alpha)}=\rr{w}_{k} \cdot s_{m+k - c_{i,k+1}(\alpha)}$.	

\item Suppose that $\alpha_{k} +c_{i,k}(\alpha) \geqslant \alpha_{i}$. Then $ c_{i,k+1}(\alpha)=  c_{i,k}(\alpha)$ and $m+k-1 - c_{i,k}(\alpha) \leqslant \alpha_{i}+k - 2 - c_{i,k}(\alpha)\leqslant \alpha_{k}+k-2$. Also notice that $c_{i,k}(\alpha) \leqslant c_{i,j}(\alpha) = \alpha_{i}' - \alpha_{j}\leqslant \alpha_{i}'$ and $m+k-1 - c_{i,k}(\alpha) \geqslant \alpha_{i}'+k-c_{i,k}(\alpha) \geqslant k$. Thus, $k\leqslant m+k-1 - c_{i,k}(\alpha)\leqslant \alpha_{k}+k-2$.
\end{itemize}

We get the following sequence of moves: $s_{m+k-1 - c_{i,k}(\alpha)}$ commutes with each simple reflection in $s_{\alpha_{k}+k-1}\cdots s_{m+k-1 - c_{i,k}(\alpha)+2}$; then we apply a braid move to get 
\begin{equation*}
s_{m+k-1 - c_{i,k}(\alpha)} \cdot s_{m+k - c_{i,k}(\alpha)} \cdot s_{m+k-1 - c_{i,k}(\alpha)}=s_{m+k- c_{i,k}(\alpha)} \cdot s_{m+k-1 - c_{i,k}(\alpha)} \cdot s_{m+k - c_{i,k}(\alpha)}
\end{equation*}
and we continue to commute $s_{m+k - c_{i,k}(\alpha)}$ with $s_{m+k-2 - c_{i,k}(\alpha)} \cdots s_{k}$. Hence, we have that $s_{m+k -1- c_{i,k}(\alpha)} \cdot \rr{w}_{k} = \rr{w}_{k} \cdot s_{m+k - c_{i,k}(\alpha)}= \rr{w}_{k} \cdot s_{m+k - c_{i,k+1}(\alpha)}$.
\end{proof}

We are now ready to prove our main theorem which states that covering relations for permutations in the strong Bruhat order are equivalent to covers for compositions in $\mathcal{PC}_n$.

\begin{thm}\label{thm:codecovering2}
Let $w,w'\in S_{n}$ and denote by $\alpha = \code(w)$ and $\alpha' = \code(w')$. Then, $w$ covers $w'$ with $w'=w\cdot (i,j)$ if, and only if, $\alpha$ covers $\alpha'$ in position $(i,j)$.
\end{thm}
\begin{proof}

Suppose that $w$ covers $w'$ such that $w = w' \cdot (i,j)$, i.e., we can write $w = w_{1}\, \cdots \, w_{i}\, \cdots\, w_{j}\, \cdots \, w_{n}$ and $w' = w_{1}\, \cdots \, w_{j}\, \cdots\, w_{i}\, \cdots \, w_{n}$ with $w_i=w'_j>w_j=w'_i$. Recall that the code of a permutation counts the number of inversions to the right of the position.

Consider the regions in the permutation matrix as in Figure \ref{fig:matrix_regions2}. Denote by $d(X)$ the number of dots in the respective region $X$.

\begin{figure}[ht]
\centering
\includegraphics[scale=1]{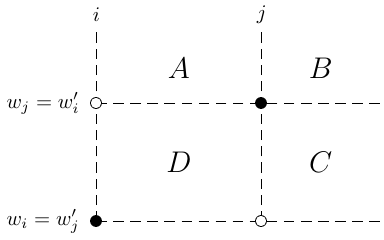}
\caption{Regions in the permutation matrix of $w$ (black dots) and $w'$ (white dots). The dots in columns $i$ and $j$ do not belong to any region.}
\label{fig:matrix_regions2}
\end{figure}

It follows from Lemma \ref{lem:bjorner} that $d(D) = 0$. Then, $\alpha_{i} = d(A) +d(B)+d(C) +1$ and $\alpha_{i}' = d(A)+d(B)$. Hence, $d(C) = \alpha_{i}-\alpha_{i}'-1 \geqslant 0$ and we have that $\alpha_{i} \geqslant \alpha_{i}' + 1$, which proves condition \eqref{eq:a1}. Observe that $\alpha_{j}'= d(B)+d(C) = \alpha_{j} + \alpha_{i}-\alpha_{i}'-1$, which proves condition \eqref{eq:a2}. 

Clearly, $\alpha_{k} = \alpha_{k}'$ for $k<i$ or $j<k$. If $i<k<j$ then it follows from Lemma \ref{lem:bjorner} that $\alpha_{k}=\alpha_{k}'$, which proves condition \eqref{eq:a3}.

Notice that $d(A) = \alpha'_{i}-\alpha_{j}$. By Proposition \ref{prop:valueD}, $c_{i,j}(\alpha)= d(A) + d(D) = d(A) = c_{i,j}(\alpha')$, which gives condition \eqref{eq:a4}.

Now, assume that $\alpha$ covers $\alpha'$ in positions $(i,j)$, where $i<j$.
We will show that for some $I\in [\ell(w)]$ there is a sequence of moves from $\rrt{w}{I}$ to $\rr{w'}$, i.e., $\rrt{w}{I}$ is a reduced decomposition of $w'$.

If $\alpha_{i} = \alpha_{i}'+1$ then $w$ covers $w'$ by Lemma \ref{lem:covering}.

Now, suppose that $\alpha_{i} > \alpha_{i}'+1$. As a consequence of conditions \eqref{eq:a1}, \eqref{eq:a2}, and \eqref{eq:a3} we have that
\begin{align*}
\rr{w}_{i}&=s_{\alpha_{i}+i-1} s_{\alpha_{i}+i-2} \cdots s_{\alpha_{i}'+i} \cdot \rr{w}_{i}'\\
\rr{w}_{j}'&= s_{\alpha_{j}'+j-1} s_{\alpha_{j}'+j-2} \cdots s_{\alpha_{j}+j} \cdot \rr{w}_{j}\\
\rr{w}_{k}&=\rr{w}_{k}', \mbox{ for $k\neq i$ and $k\neq j$.}
\end{align*}

Moreover, condition \eqref{eq:a2} also says that 
\begin{equation*}
\ell(s_{\alpha_{j}'+j-1} s_{\alpha_{j}'+j-2} \cdots s_{\alpha_{j}+j}) = \ell(s_{\alpha_{i}+i-1} s_{\alpha_{i}+i-2} \cdots s_{\alpha_{i}'+i+1}).
\end{equation*}

Choose $I =\left(\sum_{k=1}^{i} \alpha_{k}\right) - \alpha_{i}'$. We have that $\rrt{w}{I}$ and $\rr{w}'$ are written as follows:
\begin{align*}
\rrt{w}{I} &= \rr{w}_{1} \cdots \rr{w}_{i-1} \cdot ( s_{\alpha_{i}+i-1} s_{\alpha_{i}+i-2} \cdots s_{\alpha_{i}'+i+1} )\cdot \rr{w}_{i}' \cdot \rr{w}_{i+1}  \cdots \rr{w}_{n-1},\\
\rr{w'} &=\rr{w}_{1}\cdots \rr{w}_{i-1} \cdot \rr{w}_{i}'  \cdots \rr{w}_{j-1}' \cdot ( s_{\alpha_{j}'+j-1} s_{\alpha_{j}'+j-2} \cdots s_{\alpha_{j}+j} ) \cdot \rr{w}_{j} \cdots \rr{w}_{n-1}.
\end{align*}

Now, let us describe a sequence of moves that transforms $\rrt{w}{I}$ into $\rr{w'}$. For each $m\in [\alpha_{i}'+1,\alpha_{i}-1]$, we can apply Lemma \ref{prop:moves_property2} to describe the sequence of moves as follows:
\begin{align*}
s_{m+i} \cdot \rr{w}_{i}' \cdots \rr{w}_{j-1}' &= \rr{w}_{i}' \cdot s_{m+i} \cdot \rr{w}_{i+1}' \cdot \rr{w}_{i+2}' \cdots \rr{w}_{j-1}' \nonumber \\
&= \rr{w}_{i}'  \cdot \rr{w}_{i+1}' \cdot s_{m+i+1-c_{i,i+2}(\alpha)} \cdot \rr{w}_{i+2}' \cdots \rr{w}_{j-1}' \nonumber \\
&= \rr{w}_{i}'  \cdot \rr{w}_{i+1}' \cdot \rr{w}_{i+2}' \cdot s_{m+i+2-c_{i,i+3}(\alpha)}  \cdots \rr{w}_{j-1}' \label{eq:commutestep}\\
&\qquad  \vdots \nonumber\\
& = \rr{w}_{i}'  \cdots \rr{w}_{j-1}' \cdot s_{m+j-1- c_{i,j}(\alpha)} = \rr{w}_{i}'  \cdots \rr{w}_{j-1}' \cdot s_{m+j-1+\alpha_{j}-\alpha_{i}'}\nonumber
\end{align*}
where the last equality is due to condition \eqref{eq:a4}. Hence, using condition \eqref{eq:a2}, we have
\begin{align*}
\rrt{w}{I} &= \rr{w}_{1} \cdots \rr{w}_{i-1} \cdot (s_{\alpha_{i}+i-1} s_{\alpha_{i}+i-2} \cdots s_{\alpha_{i}'+i+1} )\cdot \rr{w}_{i}' \cdots \rr{w}_{j-1}' \cdot \rr{w}_{j}  \cdots \rr{w}_{n-1} \\
& = \rr{w}_{1} \cdots \rr{w}_{i-1} \cdot\rr{w}_{i}' \cdots \rr{w}_{j-1}' \cdot ( s_{\alpha_{j}'+j-1} s_{\alpha_{j}'+j-2} \cdots s_{\alpha_{j}+j} ) \cdot \rr{w}_{j}  \cdots \rr{w}_{n-1}= \rr{w'}.
\end{align*}

Hence, the reduced decompositions of $\rrt{w}{I}$ and $\rr{w'}$ produce the same permutation $w'$ (in one-line notation). By the bijection between $S_{n}$ and $\mathcal{C}_{n}$ (see Lemma \ref{lem:perm_by_diagram} and Figure \ref{fig:ex_comp_diag}) along with the uniqueness of $(i,j)$ for compositions we have $w = w'\cdot (i,j)$.
\end{proof}

\begin{cor}
The poset in $S_{n}$ defined by the Bruhat order is isomorphic to the poset $\mathcal{PC}_{n}$.
\end{cor}

\subsection{Interpretation of covering relations in the diagram}\label{subsection:diagrams}

Given two compositions $\alpha=\code(w)$ and $\alpha'=\code(w')$ such that $\alpha$ covers $\alpha'$, suppose that $D'$ is the collection of boxes obtained from the diagram of $\alpha$ where the $(\alpha_{i}'+1)$ box from the $i$-th row was removed. 

If $\alpha_i'=\alpha_i-1$ then $D'$ is already the diagram of $w'$. Now, if $\alpha_i>\alpha'_i-1$, then we must move the remaining box(es) of $D'$ at right of the removed box of until we get the diagram of $w'$. This will be made by a specific move in the $D'$ which is called a ladder move described in Figure \ref{fig:laddermove} (Bergeron-Billey in \cite{BBil93} describes a ladder move as an operation on RC-graphs). 

\begin{figure}[ht]
\includegraphics[scale=1]{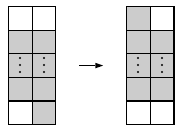}
\caption{Ladder move: both columns are adjacent and the number between the modified rows is arbitrary}
\label{fig:laddermove}
\end{figure}

\begin{prop}\label{prop:laddermoves}
The composition $\alpha$ covers $\alpha'$ in the positions $(i,j)$ if, and only if, there exists a sequence of ladder moves on $D'$ that transforms it to the diagram of $\alpha'$ by moving the remaining boxes in the $i$-th row to the $j$-th row.
\end{prop}

The proof of this proposition comes from the proof of Theorem \ref{thm:codecovering2}. The condition \eqref{eq:a4} guarantees the existence of this sequence of ladder moves.

Proposition \ref{prop:laddermoves} also resembles the excitations of boxes in the diagrams as introduced by Ikeda-Naruse \cite{IN08} in the context of the Grassmannian permutations.

For example, if $\alpha = (4,5,4,1,0,2,0)$ and $\alpha' = (2,5,4,1,0,2,1)$ then the collection of boxes $D'$ is obtained from the diagram of $\alpha$ by removing the third box in the bottom row from the diagram of $w$. The set $D'$ can be modified into $\alpha'$ through consecutive ladder moves in the diagram as we can see in Figure \ref{fig:ex_laddermove}.

\begin{figure}[ht]
\includegraphics[scale=0.75]{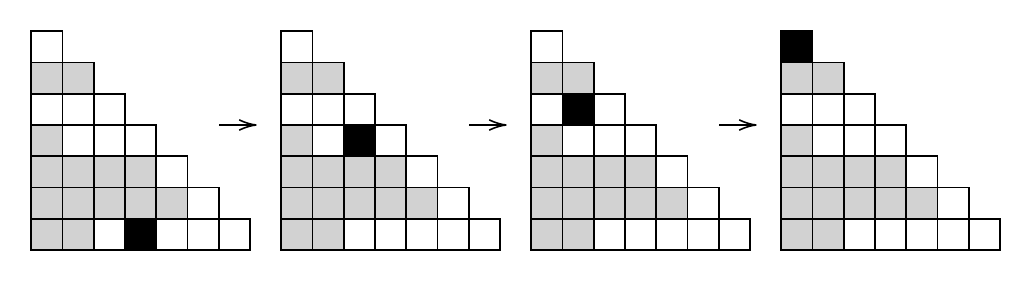}
\caption{Ladder moves applied to $D'$ result in the diagram of $\alpha' = (2,5,4,1,0,2,1)$.}
\label{fig:ex_laddermove}
\end{figure}

\section{Removable and insertable compositions}\label{sec:rem_ins_comp}

As a byproduct of the partial order in $\mathcal{C}_{n}$, this section presents an algorithm based on a removal-insertion process for determining covering relations in the poset.

\subsection{Removable compositions}

Let $\alpha\in \mathcal{C}_{n}$ be a composition.
Given $i\in [n-1]$ and $z\in [\alpha_{i}]$, we say that $\alpha$ is $(i,z)$-removable if there exists $\alpha'$ such that $\alpha$ covers $\alpha'$ and $\alpha_{i}'=\alpha_{i}-z$. We will show in the next lemma that such $\alpha'$ is unique.
Since the covering relation among $\alpha$ and $\alpha'$ requires a pair $i<j$, we now address the question of finding the corresponding $j$ for each $i$ and $z$. 

First of all, set the composition $\tilde{\alpha}$ as $\tilde\alpha_{i}=\alpha_{i}-z$ and $\tilde\alpha_{m}=\alpha_{m}$ for $m\neq i$.
Denote by $\tilde{J}_{\alpha}(i,z)$ the greatest $j>i$ such that $c_{i,j}(\alpha) = c_{i,j}(\tilde{\alpha})$, i.e., 
\begin{equation}\label{eq:tildeJ}
\tilde{J}_{\alpha}(i,z)=\max\{j>i \colon c_{i,j}(\alpha) = c_{i,j}(\tilde{\alpha})\}.
\end{equation}

Such $\tilde{J}_{\alpha}(i,z)$ exists due to Lemma \ref{lem:c_properties}. In fact, by Lemma \ref{lem:c_properties}(iii), there exists $j_1$ and $j_2$ such that $c_{i,j_1}(\alpha)=\alpha_i$ and $c_{i,j_2}(\tilde\alpha)=\tilde\alpha_i$. Since $\tilde\alpha_i < \alpha_i$, then there is a greatest $j$ such that $c_{i,j}(\alpha)=c_{i,j}(\tilde\alpha)$.
If it is clear, we will denote $\tilde{J}(i,z)=\tilde{J}_{\alpha}(i,z)$. 

\begin{lem}\label{lem:remove_uniquely}
If $\alpha$ is $(i,z)$-removable then there exists a unique composition $\alpha'$ such that $\alpha$ covers $\alpha'$ and $\alpha_{i}'=\alpha_{i}-z$. In this case, $\alpha$ covers the composition $\alpha'$ defined by $\alpha'_{i} = \alpha_{i}-z$, $\alpha'_{\tilde{J}(i,z)}=\alpha_{\tilde{J}(i,z)}+z-1$, and $\alpha'_{m}=\alpha_{m}$ for $m\neq i \mbox{ or } \tilde{J}(i,z)$.
\end{lem}
\begin{proof}
Assume that $\alpha$ covers $\alpha'$ in positions $(i,j)$ such that $\alpha_{i}'=\alpha_{i}-z$. 
We claim that $c_{i,k}(\alpha') = c_{i,k}(\tilde{\alpha})$ for any $k\in [i+1, j]$. In fact, since $\alpha'$ and $\tilde{\alpha}$ only differ for $\alpha'_{j} \geqslant \tilde{\alpha}_{j}$, it follows from Lemma \ref{lem:compare_comp} that $c_{i,k}(\alpha') \leqslant c_{i,k}(\tilde{\alpha})\leqslant c_{i,k}(\alpha)$ for any $k>i$ and $c_{i,k}(\tilde\alpha) \leqslant c_{i,k}(\alpha')$ for any $k\in [i+1, j]$.

Suppose that $\tilde{J}(i,z)<j$. It follows that $c_{i,\tilde{J}(i,z)+1}(\alpha) \neq c_{i,\tilde{J}(i,z)+1}(\tilde\alpha) = c_{i,\tilde{J}(i,z)+1}(\alpha')$ which contradicts Proposition \ref{prop:moves_property}(i). Hence $\tilde{J}(i,z)\geq j$. Notice that, by condition \eqref{eq:a4}, $\tilde\alpha_{j}=\alpha_{j}=\alpha_{i}'-c_{i,j}(\alpha')=\tilde{\alpha}_{i}-c_{i,j}(\tilde{\alpha})$. Then, $c_{i,j+1}(\tilde{\alpha})=c_{i,j}(\tilde{\alpha})= c_{i,j}(\alpha') = c_{i,j}({\alpha})\neq c_{i,j+1}({\alpha})$, which implies $j=\tilde{J}(i,z)$.
\end{proof}

Since $\alpha'$ in Lemma \ref{lem:remove_uniquely} is unique then we will call $\alpha'$ the $(i,z)$-removal of $\alpha$.

The next proposition gives an equivalent characterization for $\alpha$ be $(i,z)$-removable.

We define $k_{\alpha}(i)$ to be the first $k$ greater than $i$ such that $\alpha_{k}<\alpha_{i}$, i.e., 
\begin{equation}\label{eq:kalphai}
k_{\alpha}(i)=\min\{k>i \colon \alpha_{k}<\alpha_{i}\}.
\end{equation}

\begin{prop}\label{prop:remove_condition}
Let $\alpha$ be a composition, $i\in [n-1]$ and $z\in [\alpha_{i}]$. We have the following:
\begin{enumerate}
\item[(i)] $\alpha$ is $(i,z)$-removable if, and only if, $c_{i,\tilde{J}(i,z)}(\alpha) = \alpha_{i}-\alpha_{\tilde{J}(i,z)}-z$;
\item[(ii)] If $\alpha$ is $(i,z)$-removable then $z\leqslant\alpha_{i}-\alpha_{k_{\alpha}(i)}$;
\item[(iii)] If $\alpha_{i}>0$ then $\alpha$ is always $(i,1)$-removable.
\end{enumerate}
\end{prop}
\begin{proof}
(i): Suppose that $\alpha$ is $(i,z)$-removable. Then, $\alpha'$ as defined in Lemma \ref{lem:remove_uniquely} is covered by $\alpha$. Hence, condition \eqref{eq:a4} gives $c_{i,\tilde{J}(i,z)}(\alpha) = \alpha_{i}-\alpha_{\tilde{J}(i,z)}-z$. 

Conversely, $\alpha'$ defined by $\alpha'_{i} = \alpha_{i}-z$, $\alpha'_{\tilde{J}(i,z)}=\alpha_{\tilde{J}(i,z)}+z-1$, and $\alpha'_{m}=\alpha_{m}$ for $m\neq i \mbox{ or } \tilde{J}(i,z)$ satisfies conditions \eqref{eq:a1}, \eqref{eq:a2}, \eqref{eq:a3}, and partially \eqref{eq:a4}. It remains to prove that $c_{i,\tilde{J}(i,z)}(\alpha)=c_{i,\tilde{J}(i,z)}(\alpha')$. By Lemma \ref{lem:compare_comp}, we have that $c_{i,\tilde{J}(i,z)}(\alpha')=c_{i,\tilde{J}(i,z)}(\tilde{\alpha})=c_{i,\tilde{J}(i,z)}(\alpha)$.

(ii): By definition of $k_{\alpha}(i)$, $c_{i,k}(\alpha)=0$ for $k\in[k_{\alpha}(i)]$, $c_{i,k_{\alpha}(i)+1}(\alpha)=1$, and $\tilde{J}(i,z)\geqslant k_{\alpha}(i)$. If $\tilde{J}(i,z) = k_{\alpha}(i)$ then, by assertion (i), $z = \alpha_{i} -\alpha_{k_{\alpha}(i)}-c_{i,k_{\alpha}(i)}(\alpha) = \alpha_{i}-\alpha_{k_{\alpha}(i)}$. If $\tilde{J}(i,z)>k_{\alpha}(i)$ then, by Proposition \ref{prop:moves_property}(i), $c_{i,k_{\alpha}(i)+1}(\alpha')=c_{i,k_{\alpha}(i)+1}(\alpha)=1$ and, hence, by definition and condition \eqref{eq:a3}, $\alpha_{k_{\alpha}(i)}=\alpha'_{k_{\alpha}(i)} < \alpha_{i}' - c_{i,\alpha_{k_{\alpha}(i)}}(\alpha') = \alpha_{i}'$. Then, $z = \alpha_{i}-\alpha_{i}' < \alpha_{i}-\alpha_{k_{\alpha}(i)}$.

(iii): Let $\alpha'$ be defined by $\alpha'_{i} = \alpha_{i}-1$ and $\alpha'_{k}=\alpha_{k}$ for $k\neq i$.
By definition of $\tilde{J}(i,z)$, $c_{i,\tilde{J}(i,z)+1}(\alpha')< c_{i,\tilde{J}(i,z)+1}(\alpha)$ and, hence, $c_{i,\tilde{J}(i,z)+1}(\alpha')=c_{i,\tilde{J}(i,z)}(\alpha')$ and $c_{i,\tilde{J}(i,z)+1}(\alpha)=1+c_{i,\tilde{J}(i,z)}(\alpha)$. By definition, $\alpha_{\tilde{J}(i,z)}< \alpha_{i}- c_{i,\tilde{J}(i,z)} (\alpha) = \alpha_{i}'+1-c_{i,\tilde{J}(i,z)}(\alpha)$ and $\alpha_{\tilde{J}(i,z)}=\alpha_{\tilde{J}(i,z)}'\geqslant \alpha_{i}'- c_{i,\tilde{J}(i,z)}(\alpha')=\alpha_{i}' - c_{i,\tilde{J}(i,z)}(\alpha)$. Reordering both inequalities we have that $\alpha_{i}'-\alpha_{\tilde{J}(i,z)}\leqslant c_{i,\tilde{J}(i,z)}(\alpha)<\alpha_{i}'+1-\alpha_{\tilde{J}(i,z)}$. Hence, $c_{i,\tilde{J}(i,z)}(\alpha)=\alpha_{i}'-\alpha_{\tilde{J}(i,z)}=\alpha_{i}-\alpha_{\tilde{J}(i,z)}-1$ and, by item(i), $\alpha$ is $(i,1)$-removable.
\end{proof}

Summarizing, given a composition $\alpha \in \mathcal{C}_n$ and an index $i\in [n-1]$, we can determine $z$ such that $\alpha$ is $(i,z)$-removable by following the steps below:
\begin{enumerate}
\item Compute the possible values of $z$ by the range $[1,\alpha_i-\alpha_{k_{\alpha}(i)}]$ (cf. Equation \ref{eq:kalphai}); 
\item If $z=1$ and $\alpha_{i}>0$, then $\alpha$ is $(i,1)$-removable;
\item For each $2\leq z\leq \alpha_i-\alpha_{k_{\alpha}(i)}$, proceed as follows: 
\begin{enumerate}
\item Start with $\tilde{\alpha}$; 
\item Compute $c_{i,j}(\tilde{\alpha})$ and $\tilde{J}(i,z)$ (cf. Equation \ref{eq:tildeJ});
\item If $c_{i,\tilde{J}(i,z)}(\alpha) = \alpha_{i}-\alpha_{\tilde{J}(i,z)}-z$, then $\alpha$ is $(i,z)$-removable. Define $\alpha'$ by:
\begin{equation*}
\alpha'_{i} = \alpha_{i}-z \, , \,  \alpha'_{\tilde{J}(i,z)}=\alpha_{\tilde{J}(i,z)}+z-1, \mbox{ and }\alpha'_{m}=\alpha_{m}, \mbox{ for } m\neq i \mbox{ or } \tilde{J}(i,z),
\end{equation*}
as stated in Lemma \ref{lem:remove_uniquely}.
\end{enumerate}
\end{enumerate}

For instance, let $\alpha=(4,5,4,1,0,2,0) \in \mathcal{C}_8$. Let us determine when $\alpha$ is $(1,z)$-removable, with $z\in[\alpha_1]=[4]$. Since $k_{\alpha}(1)=4$, it follows that $z\leq \alpha_{1}-\alpha_{4}=3$, i.e., $\alpha$ is not $(1,4)$-removable. Table \ref{tab:removals} presents the $(1,z)$-removals $\alpha'$ of $\alpha$, for $z=1,2,3$. We remember that $c_{1,j}(\alpha)=(0,0,0,0,1,2,2,3,4)$. In each case, we have that $c_{i,\tilde{J}(i,z)}(\alpha) = \alpha_{i}-\alpha_{\tilde{J}(i,z)}-z$. 

\begin{table}[h!]
\centering
\renewcommand{\arraystretch}{1.4}
\begin{tabular}{>{\centering\arraybackslash}p{1.0cm}>{\centering\arraybackslash}p{3.6cm}>{\centering\arraybackslash}p{1.2cm}>{\centering\arraybackslash}p{1.2cm}>{\centering\arraybackslash}p{3.2cm}}
\toprule
${z}$ & ${c_{1,j}(\widetilde{\alpha})}$ & ${\widetilde{J}(1,z)}$ & ${\alpha'_{\widetilde{J}(1,z)}}$ & ${\alpha'}$ \\
\midrule
$1$ & $(\mathbf{0,0,0,0,1,2,2,3},3)$ & $8$ & $-$ & $(\mathbf{3},5,4,1,0,2,0)$ \\

$2$ & $(\mathbf{0,0,0,0,1,2,2},2,2)$ & $7$ & $1$ & $(\mathbf{2},5,4,1,0,2,\mathbf{1})$ \\

$3$ & $(\mathbf{0,0,0,0},0,1,1,1,1)$ & $4$ & $3$ & $(\mathbf{1},5,2,\mathbf{3},0,2,0)$ \\
\bottomrule
\end{tabular}
\caption{The $(1,z)$-removals of $\alpha=(4,5,4,1,0,2,0)$.}
\label{tab:removals}
\end{table}

Figure~\ref{fig:alpha_removings} illustrates this process using diagrams. Eventually, after a sequence of ladder moves, illustrated by the polygonal paths, one obtains the covering diagrams resulting from removing a box in the first row, which is marked by a diamond-shaped box.

\begin{figure}[ht]
\centering
\includegraphics[scale=.75]{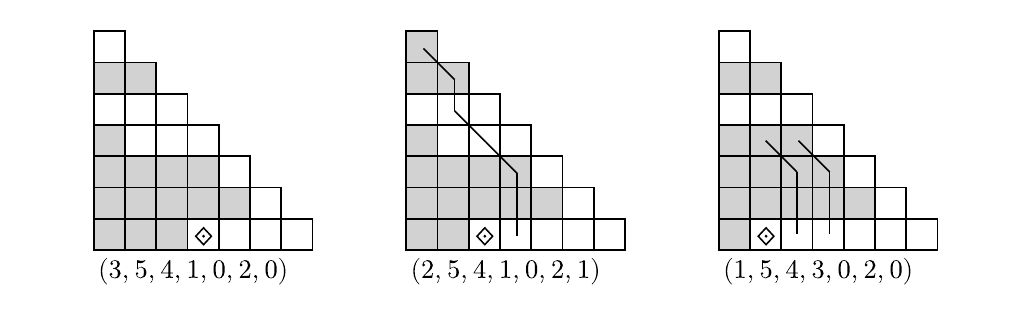}
\caption{The diagrams of the $(1,z)$-removals compositions of $\alpha=(4,5,4,1,0,2,0)$.}
\label{fig:alpha_removings}
\end{figure}

The product order of $C_{n}$ is given by $\alpha' \leqslant \alpha$ if, and only if, $\alpha_{i}'\leqslant \alpha_{i}$ for all $i\in[n-1]$.

The next corollary states that $\mathcal{PC}_{n}$ is a refinement of the product order of $C_{n}$.

\begin{cor}\label{prop:productorder}
Let $\alpha$ and $\alpha'$ be compositions in $\mathcal{C}_{n}$ such that $\alpha'\leqslant \alpha$. Then $\alpha' \preceq \alpha$ in the poset $\mathcal{PC}_{n}$.
\end{cor}
\begin{proof}
Notice that if $\alpha$ and $\alpha'$ are compositions such that $|\alpha|=|\alpha'|+1$, $\alpha_{i}' = \alpha_{i} - 1$ and $\alpha_{k}' = \alpha_{k} \mbox{ for every } k \neq i$ then $\alpha$ covers $\alpha'$.
\end{proof}

\begin{prop}\label{prop:compare_remov}
Let $\alpha$ be $(i,z_{1})$-removable and also $(i,z_{2})$-removable. Then,
\begin{enumerate}
\item[(i)] $z_{1}=z_{2}$ if, and only if, $\tilde{J}(i,z_{1})=\tilde{J}(i,z_{2})$;
\item[(ii)] $z_{1}>z_{2}$ if, and only if, $\tilde{J}(i,z_{1})<\tilde{J}(i,z_{2})$.
\end{enumerate}
\end{prop}
\begin{proof}
Let $\alpha'$ and $\alpha''$ be the compositions covered by $\alpha$ such that $\alpha'_{i}=\alpha_{i}-z_{1}$ and $\alpha''_{i}=\alpha_{i}-z_{2}$, respectively.

Clearly, if $z_{1}=z_{2}$ then $\tilde{J}(i,z_{1})=\tilde{J}(i,z_{2})$. Conversely, if $\tilde{J}(i,z_{1})=\tilde{J}(i,z_{2})=j$ then, by condition \eqref{eq:a4} of $\alpha'$ and $\alpha''$, we have that $c_{i,j}(\alpha)=\alpha_{i}'-\alpha_{j}=\alpha_{i}''-\alpha_{j}$, i.e., $\alpha_{i}'=\alpha_{i}''$. Hence $z_1=\alpha_{i}-\alpha_{i}'=\alpha_{i}-\alpha_{i}''=z_{2}$.

Suppose that $z_{1}>z_{2}$, i.e., $\alpha'_{i}<\alpha''_{i}$. 
By Lemma \ref{lem:compare_comp}, $c_{i,k}(\alpha'')\leqslant c_{i,k}(\alpha)$ for every $k>i$, and $c_{i,k}(\alpha') \leqslant c_{i,k}(\alpha'')$ for $k\in[i+1,\tilde{J}(i,z_{2})]$. By definition, $c_{i,k}(\alpha') = c_{i,k}(\alpha)$ for $k\in [i+1, \tilde{J}(i,z_{1})]$ which implies that $c_{i,k}(\alpha) = c_{i,k}(\alpha'')$ for any $k\in [i+1,\tilde{J}(i,z_{1})]$. Hence, $\tilde{J}(i,z_{1})\leqslant \tilde{J}(i,z_{2})$. By assertion (i), they must differ.

The converse follows from assertion (i) and the analogous fact if $z_{1}<z_{2}$ then $\tilde{J}(i,z_{1})>\tilde{J}(i,z_{2})$.
\end{proof}

Hence, Proposition \ref{prop:compare_remov} says that $\tilde{J}(i,\cdot)$ is an injective (\ref{prop:compare_remov}(i)) and strictly decreasing (\ref{prop:compare_remov}(ii)) function of $z$. It means that as the position of the removed box moves leftward, the index of the corresponding row decreases.

\subsection{Insertable compositions}

Given a composition $\alpha$ in $\mathcal{C}_{n}$, $i\in [n-1]$ and $z\in [n-i-\alpha_{i}]$, we say that $\alpha\in \mathcal{C}_{n}$ is $(i,z)$-insertable if there exists a composition $\alpha''\in \mathcal{C}_{n}$ such that $\alpha_{i}''=\alpha_{i}+z$ and $\alpha$ is covered by $\alpha''$.

For every composition $\alpha\in \mathcal{C}_{n}$, denote by $\alpha^{*}$ the composition in $\mathcal{C}_{n}$ such that $\alpha_{i}^{*}=n-i-\alpha_{i}$, for every $i\in [n-1]$.

\begin{lem}\label{lem:coverdual}
We have the following properties:
\begin{enumerate}
\item[(i)] For $i<j\leqslant n+1$, we have $c_{i,j}(\alpha^{*}) = j - i - 1 - c_{i,j}(\alpha)$;
\item[(ii)] If $\alpha$ covers $\alpha'$ then $\alpha^{*}$ is covered by $(\alpha')^{*}$.
\end{enumerate}
\end{lem}
\begin{proof}
We will prove (i) by induction on $j$. If $j = i + 1$ then $c_{i,i+1}(\alpha^{*})=0 = j - i - c_{i,j}(\alpha)$.

Suppose that the equation holds for $j$. There are two cases to analyze:
If $\alpha_{j}< \alpha_{i} - c_{i,j}(\alpha)$ then $\alpha^{*}_{j} = n-j-\alpha_{j} > n-j-\alpha_{i}+c_{i,j}(\alpha)=\alpha_{i}^{*}+i-j+c_{i,j}(\alpha) = \alpha_{i}^{*} - c_{i,j}(\alpha^{*}) - 1$, i.e., $\alpha^{*}_{j}\geqslant \alpha_{i}^{*} - c_{i,j}(\alpha^{*})$. By definition of $c_{i,j}$, we have that 
\begin{equation*}
c_{i,j+1}(\alpha^{*}) = c_{i,j}(\alpha^{*}) = j - i - 1 - c_{i,j}(\alpha) = j - i-c_{i,j+1}(\alpha).
\end{equation*}

If $\alpha_{j}\geqslant \alpha_{i} - c_{i,j}(\alpha)$ then $\alpha^{*}_{j} = n-j-\alpha_{j} \leqslant n-j-\alpha_{i}+c_{i,j}(\alpha)=\alpha_{i}^{*}+i-j+c_{i,j}(\alpha) = \alpha_{i}^{*} - c_{i,j}(\alpha^{*}) - 1$, i.e., $\alpha^{*}_{j}< \alpha_{i}^{*} - c_{i,j}(\alpha^{*})$. By definition of $c_{i,j}$, we have that 
\begin{equation*}
c_{i,j+1}(\alpha^{*}) = c_{i,j}(\alpha^{*}) +1= j - i - c_{i,j}(\alpha) = j - i - c_{i,j+1}(\alpha).
\end{equation*}

This proves assertion (i).

Assertion (ii) is easily obtained once we show conditions \eqref{eq:a1} to \eqref{eq:a4} using the definition of $\alpha^{*}$ and (i).
\end{proof}

The next lemma states the relationship between removable and insertable compositions.

\begin{lem}\label{lem:insertiondual}
$\alpha$ is $(i,z)$-insertable if, and only if, $\alpha^{*}$ is $(i,z)$-removable.
\end{lem}
\begin{proof}
Suppose that $\alpha$ is $(i,z)$-insertable. Then, there exists $\alpha''$ such that $\alpha$ is covered by $\alpha''$ and $\alpha''_{i}=\alpha_{i}+z$. By Lemma \ref{lem:coverdual}, $\alpha^{*}$ covers $(\alpha'')^{*}$ and $(\alpha'')_i^{*} = n - i -\alpha_{i}''=n-i-\alpha_{i}-z=\alpha^{*}_{i}-z$. The reciprocal is analogous.
\end{proof}

We now seek a criteria for a composition $\alpha$ to be $(i,z)$-insertable. A first approach is to apply the \ref{lem:insertiondual} by which it is equivalent to have $\alpha^{\ast}$ be $(i,z)$-removable.

Instead, we will define a composition $\widehat{\alpha}$ to play an analogous role to $\widetilde{\alpha}$ in the removing process. Set the composition $\widehat{\alpha}$ as $\widehat\alpha_{i}=\alpha_{i}+z$ and $\widehat\alpha_{m}=\alpha_{m}$ for $m\neq i$.
Denote by $\widehat{J}_{\alpha}(i,z)$ the greatest $j>i$ such that $c_{i,j}(\alpha) = c_{i,j}(\widehat{\alpha})$, i.e., 
\begin{equation}\label{eq:hatJ}
\widehat J_{\alpha}(i,z)=\max\{j>i \colon c_{i,j}(\alpha) = c_{i,j}(\widehat{\alpha})\}.
\end{equation}
If it is clear, we will denote $\widehat{J}(i,z)=\widehat{J}_{\alpha}(i,z)$. 

The next lemma relates $\tilde\alpha$ and $\widehat{\alpha}$. In particular, we will see that the corresponding $J$'s for the process of either removing from $\alpha^{\ast}$ or insertion in $\alpha$ are equivalent.

\begin{lem}\label{lem:equaljs} Let $\alpha \in \mathcal{C}_n$, $i\in [n-1]$ and $z\in [n-i-\alpha_{i}]$.
\begin{enumerate}
\item[(i)] $\tilde{\alpha^{*}} = (\widehat{\alpha})^{*}$;
\item[(ii)] $\widehat{J}_{\alpha}(i,z) = \tilde{J}_{\alpha^{*}}(i,z)$.
\end{enumerate}
\end{lem}
\begin{proof}
(i): If follows directly by the definitions: $(\widetilde{{\alpha}^{\ast}})_i=\alpha_i^{\ast}-z=n-i-\alpha_i-z$ and $(\widehat{\alpha})^{\ast}_{i}=n-i-(\widehat{\alpha}_i)=n-i-\alpha_i-z$. For $m\neq i$, $(\widetilde{{\alpha}^{\ast}})_m=\alpha_m^{\ast}=n-i-\alpha_m$ and $(\widehat{\alpha})^{\ast}_m=n-i-\widehat{\alpha}_m=n-i-\alpha_m$.

(ii): By item (i), notice that $c_{i,j}(\alpha^{\ast})=c_{i,j}(\widetilde{\alpha^{\ast}})=c_{i,j}((\widehat{\alpha})^{*})$. Hence, by Lemma \ref{lem:coverdual}(i),  $c_{i,j}(\alpha^{\ast})=c_{i,j}(\widetilde{\alpha^{\ast}})$ if, and only if,  $c_{i,j}(\alpha)=c_{i,j}(\widehat{\alpha})$. 
\end{proof}

\begin{lem}\label{lem:insertion_uniquely}
If $\alpha$ is $(i,z)$-insertable then there exists a unique composition $\alpha''$ such that $\alpha$ is covered by $\alpha''$ and $\alpha_{i}''=\alpha_{i}+z$. In this case, $\alpha$ is covered by the composition $\alpha''$ defined by $\alpha''_{i} = \alpha_{i}+z$, $\alpha''_{\widehat{J}(i,z)}=\alpha_{\widehat{J}(i,z)}-z+1$, and $\alpha''_{m}=\alpha_{m}$ for $m\neq i \mbox{ or } \widehat{J}(i,z)$.
\end{lem}
\begin{proof}
By Lemma \ref{lem:insertiondual}, $\alpha^{*}$ is $(i,z)$-removable. Then, it follows from Lemma \ref{lem:remove_uniquely} that the unique $\alpha'$ such that $\alpha^{*}$ covers $\alpha'$ is defined by $\alpha'_{i}=\alpha_{i}^{*}-z$, $\alpha'_{\tilde{J}_{\alpha^{*}}(i,z)} = \alpha_{\tilde{J}_{\alpha^{*}}(i,z)}+z-1$ and $\alpha'_{m}=\alpha^{*}$ for $m\neq i$ or $\tilde{J}_{\alpha^{*}}(i,z)$. Put $\alpha''=(\alpha')^{*}$ and notice that, by Lemma \ref{lem:equaljs}, $\tilde{J}_{\alpha^{*}}(i,z)=\widehat{J}_{\alpha}(i,z)$.
\end{proof}

Since $\alpha''$ in Lemma \ref{lem:insertion_uniquely} is unique, we will call $\alpha''$ the $(i,z)$-insertion of $\alpha$.

The next lemma shows that we can also lower the upper bound of $z$ in the insertion process.

\begin{prop}\label{prop:insertion_condition}
Let $\alpha\in \mathcal{C}_{n}$, $i\in [n-1]$, and $z\in [n-i-\alpha_{i}]$. We have the following:
\begin{enumerate}
\item[(i)] $\alpha$ is $(i,z)$-insertable if, and only if, $c_{i,\widehat{J}(i,z)}(\alpha) = \alpha_{i}-\alpha_{\widehat{J}(i,z)}+z-1$;
\item[(ii)] If $\alpha$ is $(i,z)$-insertable then $z\leqslant \alpha_{k_{\alpha^{*}}(i)}-\alpha_{i}+k_{\alpha^{*}}(i)-i$;
\item[(iii)] If $\alpha_{i}+i<n$ then $\alpha$ is always $(i,1)$-insertable.
\end{enumerate}
\end{prop}
\begin{proof}
(i): By Lemma \ref{lem:insertiondual} and Proposition \ref{prop:remove_condition}, $\alpha^{*}$ is $(i,z)$-removable if, and only if, $c_{i,\tilde{J}_{\alpha^{*}}(i,z)}(\alpha^{*}) = \alpha_{i}^{*}-\alpha^{*}_{\tilde{J}_{\alpha^{*}}(i,z)}-z$. By Lemmas \ref{lem:coverdual} and \ref{lem:equaljs}, the latter equation is equivalent to $c_{i,\widehat{J}(i,z)}(\alpha) = \alpha_{i}-\alpha_{\widehat{J}(i,z)}+z-1$.

(ii): By Lemma \ref{lem:insertiondual}, $\alpha^{*}$ is $(i,z)$-removable. Since $k_{\alpha^{*}}(i) = \min\{k>i \colon \alpha_{k}^{*}<\alpha_{i}^{*}\}$ then $z\leqslant \alpha_{i}^{*}-\alpha_{k_{\alpha^{*}}(i)}^{*} = n-i-\alpha_{i}-(n-k_{\alpha^{*}}(i)-\alpha_{k_{\alpha^{*}}(i)})=\alpha_{k_{\alpha^{*}}(i)}-\alpha_{i}+k_{\alpha^{*}}(i)-i$.

(iii): Under this condition we have that $\alpha^{*}_{i} > 0$. Then, $\alpha^{*}$ is $(i,1)$-removable, i.e., $\alpha$ is $(i,1)$-insertable.
\end{proof}

\begin{prop}\label{prop:compare_insertion}
Let $\alpha$ be $(i,z_{1})$-insertable and also $(i,z_{2})$-insertable. Then,
\begin{enumerate}
\item[(i)] $z_{1}=z_{2}$ if, and only if, $\widehat{J}(i,z_{1})=\widehat{J}(i,z_{2})$;
\item[(ii)] $z_{1}>z_{2}$ if, and only if, $\widehat{J}(i,z_{1})<\widehat{J}(i,z_{2})$.
\end{enumerate}
\end{prop}
\begin{proof}
It follows directly by Lemmas \ref{lem:insertiondual} and \ref{lem:equaljs} and by Proposition \ref{prop:compare_remov}.
\end{proof}

Summarizing, given a composition $\alpha \in \mathcal{C}_n$ and an index $i\in [n-1]$, we can determine if $\alpha$ is $(i,z)$-insertable by following the steps below:

\begin{enumerate}
\item Start with $\alpha^{\ast}$; 
\item Compute the possible values of $z$ by the range $[1,\alpha_{k_{\alpha^{*}}(i)}-\alpha_{i}+k_{\alpha^{*}}(i)-i]$ (cf. Equation \ref{eq:kalphai}); 
\item If $\alpha_{i}+i<n$ then $\alpha$ is $(i,1)$-insertable;
\item For each $2\leq z\leq \alpha_{k_{\alpha^{*}}(i)}-\alpha_{i}+k_{\alpha^{*}}(i)-i$, proceed as follows: 
\begin{enumerate}
\item Start with $\widehat{\alpha}$; 
\item Compute $c_{i,j}(\widehat{\alpha})$ and $\widehat{J}(i,z)$ (cf. Equation \ref{eq:hatJ});
\item If $c_{i,\widehat{J}(i,z)}(\alpha) = \alpha_{i}-\alpha_{\widehat{J}(i,z)}+z-1$, then $\alpha$ is $(i,z)$-insertable. Define $\alpha''$ by:
\begin{equation*}
\alpha''_{i} = \alpha_{i}+z \,,\, \alpha''_{\widehat{J}(i,z)}=\alpha_{\widehat{J}(i,z)}-z+1, \mbox{ and }\alpha''_{m}=\alpha_{m}, \mbox{ for }m\neq i \mbox{ or } \widehat{J}(i,z),
\end{equation*}
as stated in Lemma \ref{lem:insertion_uniquely}.
\end{enumerate}
\end{enumerate}

For instance, let $\alpha=(2,5,4,1,0,2,1) \in \mathcal{C}_8$. Let us determine when $\alpha$ is $(1,z)$-insertable, with $z\in[5]$. Notice that $\alpha^{\ast}=(5,1,1,3,3,0,0)$ and $k_{\alpha^{\ast}}(1)=\min\{k>1\colon \alpha^{\ast}_k <5\}=2$. Hence, $z\leq \alpha_2-\alpha_1+k_{\alpha^{\ast}}(1)-1=4$, i.e., $\alpha$ is not $(1,5)$-insertable. Table \ref{tab:insertions} presents the $(1,z)$-insertions $\alpha''$ of $\alpha$, for $z=1,\ldots,4$. We remember that $c_{1,j}(\alpha)=(0,0,0,0,1,2,2,2,2)$. In each case, we have that $c_{i,\widehat{J}(i,z)}(\alpha) = \alpha_{i}-\alpha_{\widehat{J}(i,z)}+z-1$.

\begin{table}[h!]
\centering
\renewcommand{\arraystretch}{1.4}
\begin{tabular}{>{\centering\arraybackslash}p{1.0cm}>{\centering\arraybackslash}p{3.6cm}>{\centering\arraybackslash}p{1.2cm}>{\centering\arraybackslash}p{1.2cm}>{\centering\arraybackslash}p{3.2cm}}
\toprule
${z}$ & ${c_{1,j}(\widehat{\alpha})}$ & ${\widehat{J}(1,z)}$ & ${\alpha''_{\widehat{J}(1,z)}}$ & ${\alpha''}$ \\
\midrule
$1$ & $(\mathbf{0,0,0,0,1,2,2,2},3)$ & $8$ & $-$ & $(\mathbf{3},5,4,1,0,2,1)$ \\

$2$ & $(\mathbf{0,0,0,0,1,2,2},3,4)$ & $7$ & $0$ & $(\mathbf{4},5,4,1,0,2,\mathbf{0})$ \\

$3$ & $(\mathbf{0,0,0},1,2,3,3,4,5)$ & $3$ & $2$ & $(\mathbf{5},5,\mathbf{2},1,0,2,1)$ \\

$4$ & $(\mathbf{0,0},1,2,3,4,4,5,6)$ & $2$ & $2$ & $(\mathbf{6,2},4,1,0,2,1)$ \\
\bottomrule
\end{tabular}
\caption{The $(1,z)$-insertions of $\alpha=(2,5,4,1,0,2,1)$.}
\label{tab:insertions}
\end{table}

Figure~\ref{fig:alpha_insertions} shows this process using diagrams. Eventually, after a sequence of inverse of ladder moves, illustrated by the polygonal paths, one produces the covering diagrams formed by inserting a box in the first row, marked by a diamond-shaped box.

\begin{figure}[ht]
\centering
\includegraphics[scale=0.75]{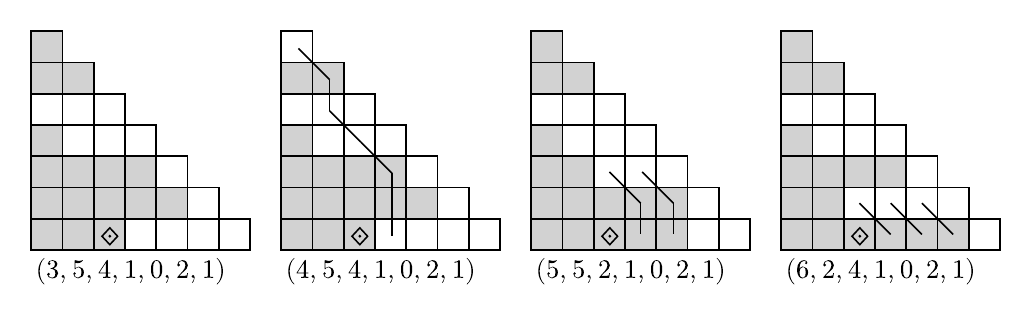}
\caption{The diagrams of the $(1,z)$-insertions of $\alpha=(2,5,4,1,0,2,1)$.}
\label{fig:alpha_insertions}
\end{figure}

\subsection{Monk's rule}\label{subsec:monksrule}

Coşkun-Taşkın (see \cite{CT18}, Theorem 4.1) obtained a result analogous to Theorem \ref{thm:codecovering2} (or Proposition \ref{prop:laddermoves}) using  the language of Tower diagrams constructed from the inverse permutation Lehmer code. As an application, their approach yields an algorithm that describes Monk's rule, which in turn can be used to derive an algorithm for Pieri's rule.

Similarly, Theorem \ref{thm:codecovering2} provides an insertion process on compositions that leads to an algorithm for computing  Monk's Rule for permutations (see Lascoux-Schützenberger \cite{LS82} for the Schubert polynomials theory). Specifically, for a permutation $w\in S_{\infty}$, let $\mathfrak{S}_{w}$ denote the corresponding Schubert polynomial. Then:
\begin{equation*}
\mathfrak{S}_{s_{r}} \mathfrak{S}_{w} = \sum_{\substack{i\leqslant r <j \\ \ell(w\cdot(i,j))=\ell(w)+1}} \mathfrak{S}_{w\cdot(i,j)}.
\end{equation*}

Define $\mathcal{C}_{\infty} = \bigcup \mathcal{C}_{n}$. In the context of compositions, given $\alpha\in\mathcal{C}_{\infty}$ and $r>0$, we need to find all $(i,z)$ such that $\alpha$ is $(i,z)$-insertable and $i\leqslant r < \widehat{J}(i,z)$. Note that Proposition \ref{prop:insertion_condition}(i) provides an explicit criterion to determine whether $\alpha$ is insertable, while Propositions \ref{prop:insertion_condition}(ii) and \ref{prop:compare_insertion} further restrict the possible values of $z$.

For instance, let $w=37621854 \in S_8$ for which we have $\code(w)=\alpha=(2,5,4,1,0,2,1)$. Let us compute the Monk's Rule for $r=4$. Table \ref{tab:monksrule} displays the possible possible insertions $(i,z)$ with $i\leq 4$, along with the corresponding values of $\widehat{J}$. Only those satisfying $\widehat{J}\geq 5$ contribute to the result.  

\begin{table}[h!]
\centering
\renewcommand{\arraystretch}{1.4}
\begin{tabular}{>{\centering\arraybackslash}p{0.6cm}>{\centering\arraybackslash}p{3.6cm}>{\centering\arraybackslash}p{0.6cm}>{\centering\arraybackslash}p{3.6cm}>{\centering\arraybackslash}p{0.6cm}>{\centering\arraybackslash}p{3.2cm}}
\toprule
$i$ & $c_{i,j}(\alpha)$ & $z$ & $c_{i,j}(\widehat{\alpha})$ & $\widehat{J}$ & $\alpha''$ \\
\midrule
\multirow{4}{*}{$1$} & \multirow{4}{*}{$(0,0,0,0,1,2,2,2,3)$} & $1$ & $(0,0,0,0,1,2,2,2,3)$ & $8$ & $(3,5,4,1,0,2,1)$\\

    &     & $2$ & $(0,0,0,0,1,2,2,3,4)$ & $7$ & $(4,5,4,1,0,2,0)$\\

    &     & $3$ & $(0,0,0,1,2,3,3,4,5)$ & $3$ & --\\

    &     & $4$ & $(0,0,1,2,3,4,4,5,6)$ & $2$ & --\\
\hline
$2$ & $(0,0,0,1,2,3,3,4,5)$ & $1$ & $(0,0,0,1,2,3,4,5,6)$ & $6$ & $(2,6,4,1,0,2,1)$ \\
\hline
$3$ & $(0,0,0,0,1,2,2,3,4)$ & $1$ & $(0,0,0,0,1,2,3,4,5)$ & $6$ & $(2,5,5,1,0,2,1)$\\
\hline
\multirow{3}{*}{$4$} & \multirow{3}{*}{$(0,0,0,0,0,1,1,1,1)$} & $1$ & $(0,0,0,0,0,1,1,1,2)$ & $8$ & $(2,5,4,2,0,2,1)$\\

    &     & $2$ & $(0,0,0,0,0,1,1,2,3)$ & $7$ & $(2,5,4,3,0,2,0)$ \\

    &     & $3$ & $(0,0,0,0,0,1,2,3,4)$ & $6$ & $(2,5,4,4,0,0,1)$\\
\bottomrule
\end{tabular}
\caption{The Monk's Rule for $\alpha=(2,5,4,1,0,2,1)$ and $r=4$.}
\label{tab:monksrule}
\end{table}

Hence,
\begin{align*}
\mathfrak{S}_{s_{4}} \mathfrak{S}_{(2,5,4,1,0,2,1)} &= \mathfrak{S}_{(3,5,4,1,0,2,1)}+\mathfrak{S}_{(4,5,4,1,0,2,0)}+\mathfrak{S}_{(2,6,4,1,0,2,1)}\\
&+\mathfrak{S}_{(2,5,5,1,0,2,1)}+\mathfrak{S}_{(2,5,4,2,0,2,1)}+\mathfrak{S}_{(2,5,4,3,0,2,0)}+\mathfrak{S}_{(2,5,4,4,0,0,1)}.
\end{align*}

\bibliographystyle{amsplain} 
\bibliography{biblio}

\end{document}